\documentclass{amsart}

\makeatletter
\@addtoreset{equation}{section}\makeatother
\newtheorem{theo}{Theorem}[section]
\newtheorem{lem}[theo]{Lemma}
\newtheorem{prop}[theo]{Proposition}

\theoremstyle{remark} \newtheorem{remark}[theo]{Remark}
\newtheorem{defn}[theo]{Definition}

\usepackage{amssymb,xcolor}
\usepackage{amscd}
\usepackage{verbatim,ifthen,epsfig}

\usepackage{amsmath,graphics}
\usepackage{latexsym}
\usepackage{eucal}

\usepackage{latexsym,url}

\def\RR{\mathbb{R}}
\def\CC{\mathbb{C}}

\def\NN{\mathbb{N}}

\def\supp{\operatorname{supp}}

\def\ang#1{{\langle #1 \rangle }}
\def\eps{\epsilon}

\def\diag{\operatorname{{diag}}}
\def\tE{\tilde E}
\renewcommand\Re{\operatorname{Re}}
\renewcommand\Im{\operatorname{Im}}

\def\Id{\operatorname{Id}}

\def\ang#1{{\langle #1 \rangle }}
\def\bang#1{{\big\langle #1 \big\rangle }}

\newcommand{\norm}[1]{\left|\!\left|{#1}\right|\!\right|}
\newcommand\spec{{\operatorname{spec}}}
\newcommand\umin{u_{\operatorname{min}}}

\newcommand\tttm{\tilde t_h}
\newcommand\Deltab{\Delta_{\partial \Omega}}

\newcommand\Omegabar{\overline{\Omega}}

\newcommand\dOmega{\partial \Omega}
\newcommand\dbar{\overline{\partial}}

\newcommand\Cd{C_{\delta}(\partial \Omega)}
\newcommand\gd{g_{\partial}}
\newcommand\ud{u^{\partial\Omega}}

\newcommand\pO{{\partial\Omega}}
\newcommand{\be}{\begin{equation}}
\newcommand{\ee}{\end{equation}}

\newcommand{\tbox}[1]{{\mbox{\rm \tiny #1}}}
\DeclareMathOperator{\Span}{span}
\newcommand{\eclas}{\varepsilon_\tbox{clas}}    
\newcommand{\enew}{\varepsilon_\tbox{new}}

\begin{document}
\title[Neumann eigenfunctions at the boundary] {Comparable upper and lower bounds for boundary values of Neumann eigenfunctions and tight inclusion of eigenvalues}

\author{Alex H. Barnett}
\address{Department of Mathematics, Dartmouth College, Hanover, NH, 03755, USA}
\email{ahb@math.dartmouth.edu}
\author{Andrew Hassell} %
\address{Department of Mathematics, Australian National University, Canberra ACT 0200, AUSTRALIA}
\email{Andrew.Hassell@anu.edu.au}

\author{Melissa Tacy} %
\address{Department of Mathematics, Adelaide University, Adelaide, SA 5005, AUSTRALIA}

\email{melissa.tacy@adelaide.edu.au}


\begin{abstract}
For smooth bounded domains in $\RR^n$,
we prove upper and lower $L^2$ bounds
on the boundary data of Neumann eigenfunctions,
and prove quasi-orthogon\-ality of this boundary data in a spectral window.
The bounds are tight in the sense that both are independent of eigenvalue;
this is achieved by working with 
an appropriate norm for boundary functions, which includes a `spectral weight', that is, a
function of the boundary Laplacian. This spectral weight is chosen to cancel concentration
at the boundary that can happen for `whispering gallery' type eigenfunctions. 
These bounds are closely related to wave equation estimates due to Tataru.

Using this,
we bound the distance from an arbitrary Helmholtz parameter $E>0$ to the nearest Neumann eigenvalue, in terms of boundary normal-derivative data of a trial function $u$ solving the Helmholtz equation $(\Delta-E)u=0$.
This `inclusion bound' improves over previously known bounds by a factor of $E^{5/6}$.
It is analogous to a recently improved inclusion bound in the Dirichlet case, due to the first two authors. 

Finally, we apply our theory to present an improved numerical
implementation of the method of particular solutions
for computation of Neumann eigenpairs on smooth planar domains. 
We show that the new inclusion bound improves the relative accuracy in 
a computed Neumann eigenvalue (around the $42000$th)
from 9 digits to 14 digits, with little extra effort.

%

\end{abstract}

\thanks{Research of A. B. was partially supported by the NSF on grants DMS-0811005 and DMS-1216656, and
research of A. H.  was supported by the Australian Research Council
through grants FT0990895, DP1095448, DP120102019 and DP150102419.
A. B. acknowledges the hospitality of the Mathematical Sciences, ANU, in 2013 and 2015,
and A. H. and M. T. acknowledge the hospitality of Dartmouth College during a
visit in 2010.}


\maketitle

\tableofcontents

\section{Introduction}
In this paper we consider Laplace eigenfunctions on a smooth bounded domain $\Omega \subset \RR^n$. As is well known, the positive Laplacian\footnote{Note that our sign convention is opposite to that of \cite{bnds}.},
$$
\Delta = - \sum_{i=1}^n \frac{\partial^2}{\partial x_i^2},
$$
with domain either $H^2(\Omega) \cap H^1_0(\Omega)$ or $$
\{ u \in H^2(\Omega) \mid d_n u |_{\partial \Omega} = 0 \}
$$
is self-adjoint. Here and below, $d_n$ denotes the directional derivative with respect to the inward unit normal vector at the boundary, $\partial \Omega$. These are known as the Laplacian with Dirichlet, resp.\ Neumann, boundary conditions and will be denoted $\Delta^D$, resp.\ $\Delta^N$. In either case, there is an orthonormal basis $u_j$ of $L^2(\Omega)$ consisting of real eigenfunctions, with eigenvalues $E_j \to \infty$.
It will usually be clear from context whether we are considering Dirichlet or Neumann eigendata, but when necessary we will write $u^D_j$, resp.\ $u^N_j$ for Dirichlet, resp.\ Neumann eigenfunctions, and similarly $E^D_j$, resp.\ $E^N_j$.
This paper presents new results for the Neumann case (two such eigenfunctions are shown in Fig.~\ref{f:intro}(a--b)).
Recall that in acoustics applications,
the wavenumbers $\sqrt{E_j^N}$ can be interpreted as resonant frequencies of a closed sound-hard cavity.
The myriad other applications and properties of 
Laplace eigenfunctions and eigenvalues are reviewed in \cite{KS,grebenkov},
while numerical solution methods are reviewed in \cite{KS,ungerbook,sca}.

\begin{figure}[t] 
\mbox{\hspace{-3ex}%
\raisebox{-.7in}{\includegraphics[width=1.2in]{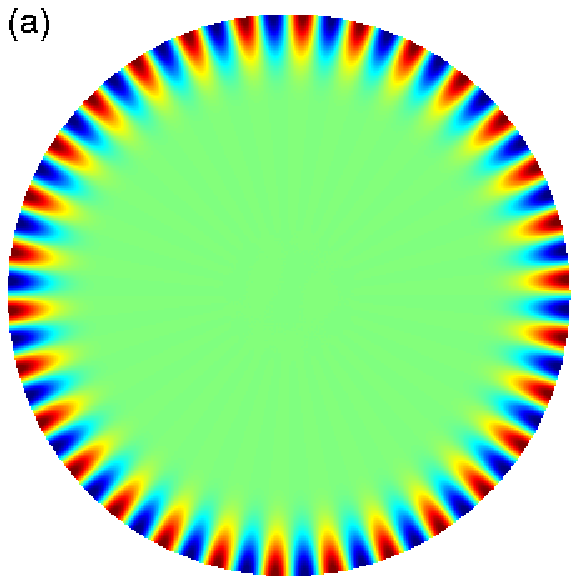}}
\raisebox{-1in}{\includegraphics[width=1.7in]{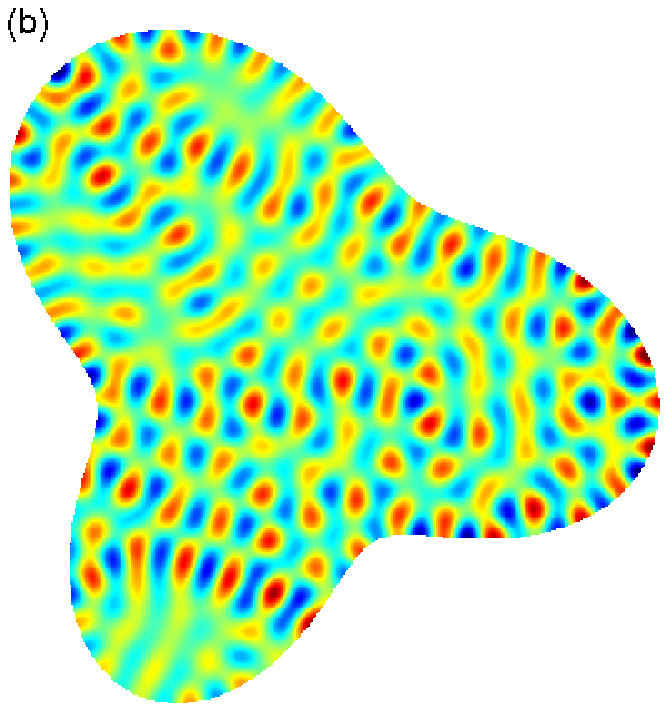}}
\raisebox{-1.3in}{\includegraphics[width=2.4in]{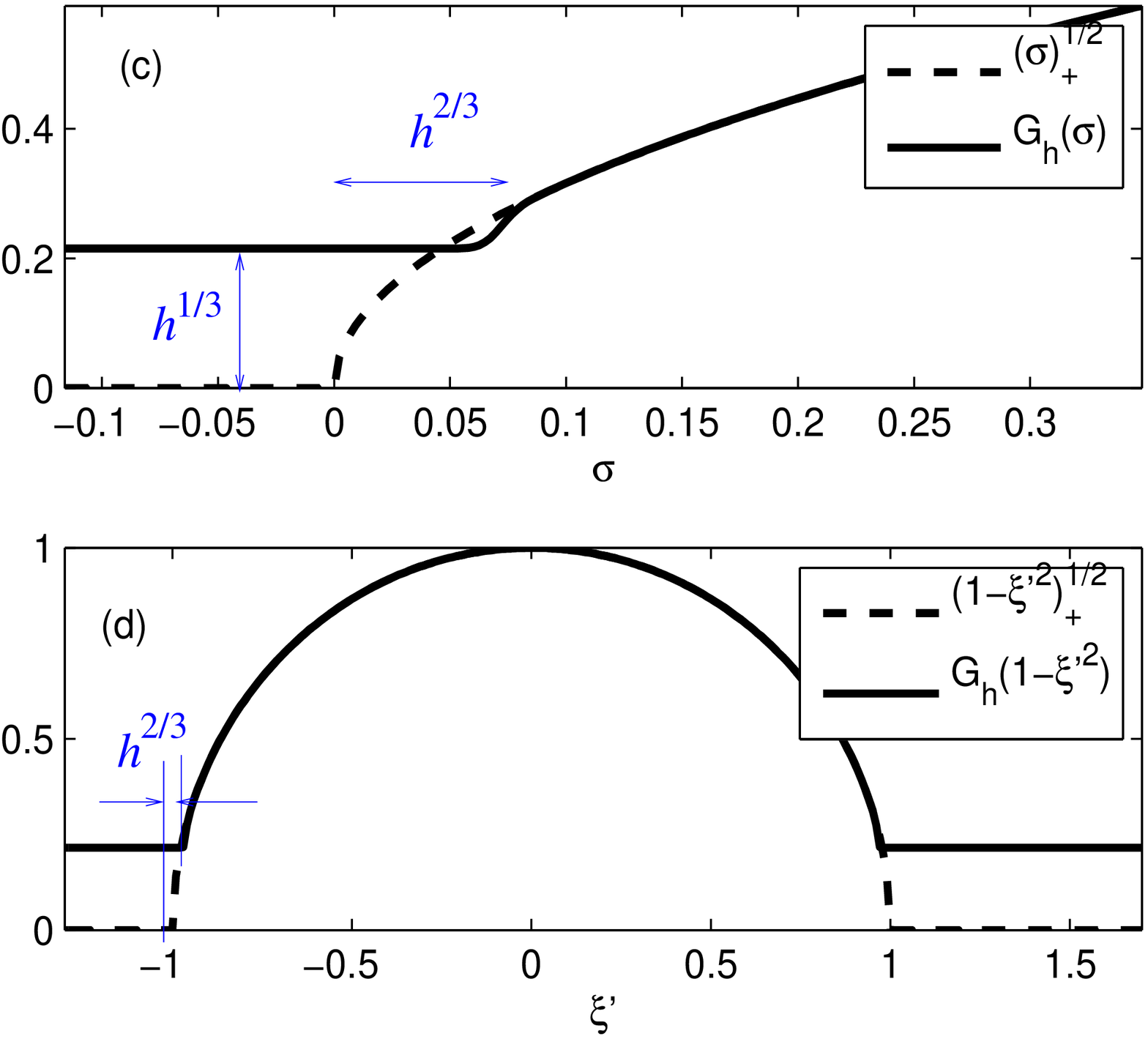}}
}%
\caption{(a) An example Neumann disc eigenmode
$u = v_{30,1}$ of the form with maximal boundary norm $\|u\|_{L^2(\partial\Omega)}$,
with eigenfrequency $\sqrt{E} =\mu_{30,1} = 32.534223556790\cdots$.
(b) Neumann eigenmode of generic nonsymmetric smooth planar domain
with eigenfrequency $\sqrt{E_j} = 40.512821995008\cdots$.
In (a) and (b), red is positive, green zero, and blue negative.
(c) Smooth regularized positive square-root function $G_h(\sigma)$
defined in \eqref{Gh}, showing $h$ scalings.
(d) The spectral weight $G_h(1-\xi'^2)$ vs transverse scaled
frequency $\xi'\in\mathbb{R}$ in the two-dimensional case.
\label{f:intro}
}
\end{figure}

\subsection{Heuristic discussion}

The equation $(\Delta - E) u = 0$ on $\Omega$ is known as the Helmholtz equation. We will refer to a solution (with no prescribed behaviour at the boundary) as a Helmholtz function of energy $E$, or frequency $\sqrt{E}$. If a Helmholtz function additionally satisfies the boundary condition $u |_{\dOmega} = 0$, resp.\ $d_n u |_{\dOmega} = 0$, then it is a Dirichlet, resp.\ Neumann eigenfunction of eigenvalue $E$. We write inverse frequency $h_j = E_j^{-1/2}$, to conform with semiclassical notation that will be used in the body of the paper. 
We will denote the spectrum of $\Delta^D$, resp.\ $\Delta^N$ by $\spec^D$, resp.\ $\spec^N$.
Also, we denote the restriction of $u \in C(\overline{\Omega})$ to the boundary by $\ud$. 

It is well known that the normal derivatives $d_n u^D_j$ of Dirichlet eigenfunctions satisfy the following upper and lower estimates:
\begin{equation}
c \sqrt{E^D_j}  \leq\| d_n u^D_j \|_{L^2(\dOmega)} \leq C\sqrt{E^D_j} ,
\label{Dir-bound}\end{equation}
where $c, C$ are independent of $j$. (Here and below, all constants will be independent of the eigenvalue.) The estimates \eqref{Dir-bound} are straightforward to prove using Rellich identities; see \cite{HT}. In \cite{bnds} the first two authors proved a strengthened version of this inequality, and used it to prove a \emph{Dirichlet inclusion bound} for Helmholtz functions. Let us explain what this means. 
Let $u$ be a smooth function on $\Omega$, continuous on the closure of $\Omega$, and not identically zero. Then we define the \emph{tension} $t[u]$ by
\begin{equation}
t[u] = \frac{\| u \|_{L^2(\dOmega)}}{\| u \|_{L^2(\Omega)}}.
\label{t-def}\end{equation}
(For notational simplicity we write $\| u \|_{L^2(\dOmega)}$ instead of $\| \ud \|_{L^2(\dOmega)}$, etc.) 
Barnett \cite{incl}, followed by the first two authors
\cite{bnds}, proved the following bound, termed an inclusion bound as it specifies an interval around $E$ that includes a point of $\spec^D$.

\begin{theo}\label{MPS-Dir} There exist positive constants $c, C$ depending only on $\Omega$ such that the following holds.  Let $u$ be any nonzero solution of $(\Delta - E) u = 0$ in $C^\infty(\Omegabar)$, and let $\umin$ be the Helmholtz solution minimizing $t[u]$ (such a minimizer exists --- see \cite[Lemma 3.1]{bnds}). Then
$$
c \sqrt{E}  t[\umin] \leq d(E, \spec^D) \leq C \sqrt{E} t[u].
$$
\end{theo}

\begin{remark} This improves on the classical Moler--Payne bound \cite{molerpayne} by a factor of $\sqrt{E}$. Moreover, it is tight in the sense that the same power of $E$ appears in the lower and upper bound. 
\end{remark}

\begin{figure}[t] 
\mbox{%
\hspace{-4ex}
\includegraphics[height=.87in]{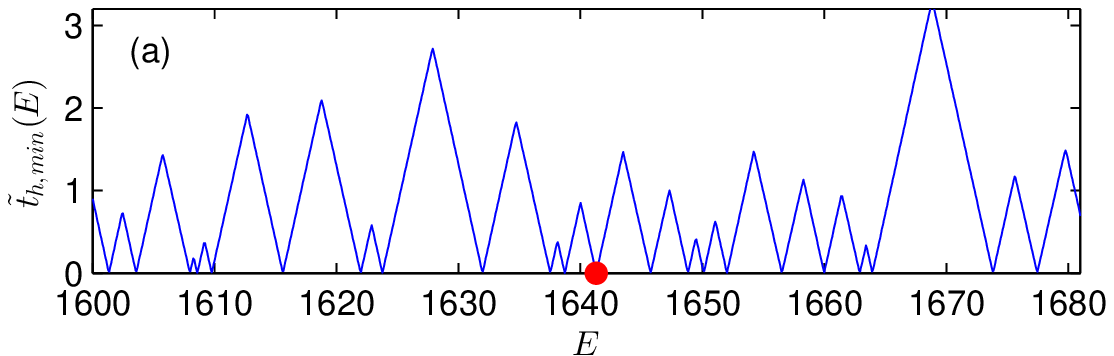}
\includegraphics[height=.87in]{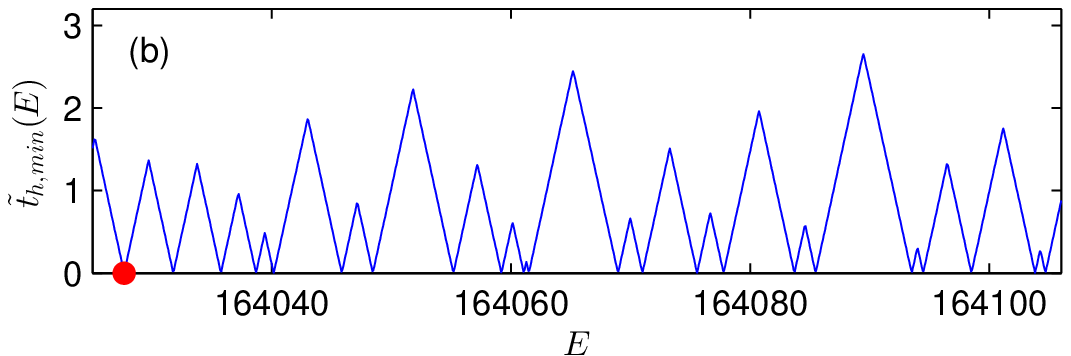}
}%
\caption{Minimum tension $\tttm$ defined by \eqref{tttm-def}
achievable at each energy parameter $E$,
for the nonsymmetric smooth planar domain
of Fig.~\ref{f:intro}(b).
(This is numerically approximated by $\tilde{t}_{h,\tbox{min}}(E)$
defined by \eqref{thmin}.)
Each local tension minimum is very close to zero,
and indicates a Neumann eigenvalue $E^N_j$.
(a) Medium frequency $\sqrt{E} \in [40, 41]$, with the large dot corresponding
to the eigenmode shown in Fig.~\ref{f:intro}(b).
(b) High frequency $\sqrt{E} \in [405, 405.1]$, with the dot corresponding
to the eigenmode of Fig.~\ref{f:high}.
The $E$ intervals in (a) and (b) have the same width, and the vertical scales
are identical, highlighting the similarity of the magnitude of the slopes
(Theorem~\ref{MPS-Neu}).
\label{f:tsweep}
}
\end{figure}

In part, the interest in inclusion bounds comes from numerical analysis. One way to numerically approximate Dirichlet eigenvalues and eigenfunctions is the \emph{method of particular solutions} (MPS) \cite{fhm,molerpayne,mps,gsvd}. In this method, one chooses an energy $E$ and then numerically minimizes $t[u]$ over all $E$-Helmholtz functions $u$. One then moves along the $E$-axis, searching for near-zeros (`roots') of this minimal value of the tension (see Fig.~\ref{f:tsweep}). Then, given an approximate root $E$ and its corresponding Helmholtz function $u$, the inclusion bound 
bounds the eigenvalue error, i.e.\ how far $E$ is from the Dirichlet spectrum.

In this paper, we seek to prove the analogous results for Neumann eigenfunctions. However, the meaning of `analogous' here is far from obvious. The first issue is that the naive analogy of \eqref{Dir-bound}, is false: that is, it is not true that  $\| u^N_j \|_{L^2(\dOmega)}$ is comparable to $1$, as $j$ tends to infinity. In fact, consideration of eigenfunctions on the circle shows that there is a lower bound of $1$, but the sharp upper bound on $\| u^N_j \|_{L^2(\dOmega)}$ is $C E_j^{1/6}$, or, now switching to semiclassical notation, where we write the Helmholtz equation in the form $(h^2 \Delta - 1) u = 0$, by  $C h_j^{-1/3}$. 

However, we can obtain a better analogy of \eqref{Dir-bound} by realizing that, in the high-frequency limit, the Neumann analogue of $h_j d_n u^D_j$ at the boundary is not simply the restriction of $u_j^N$ to the boundary. In fact, the symbol (in the semiclassical sense) of the normal derivative operator $h d_n$ is $i\xi_n$, where for $p \in \dOmega$,  $\xi_n$ is the linear function on  $T_{p} \Omega$ that vanishes on vectors tangent to $\dOmega$ and equals $1$ on the inward-pointing unit normal.\footnote{The semiclassical  symbol of a differential operator $\sum_\alpha v_\alpha(x) (hD_x)^\alpha$ on $\RR^n$, where $D_x^\alpha := \Pi_{i=1}^n (-i \partial_{x_i})^{\alpha_i}$, is equal to $\sum_\alpha v_\alpha(x) \xi_1^{\alpha_1} \dots \xi_n^{\alpha_n}$. This incorporates the semiclassical scaling, in which frequency is scaled by a factor of $h$, i.e. $\xi$ represents the oscillation $e^{ix \cdot \xi/h}$ rather than $e^{ix \cdot \xi}$. On a manifold, the semiclassical symbol is invariantly defined on the cotangent bundle.}  In the semiclassical limit, $h \to 0$, eigenfunctions are concentrated in frequency at the set $|\xi| = 1$, i.e. heuristically they are composed of a superposition of oscillations $e^{ix \cdot \xi/h}$ with frequencies $\xi$ of different directions and unit length. 
At the boundary, we can decompose such a $\xi$ into tangential and normal components, $\xi= (\xi', \xi_n)$, thus  we have $\xi_n = \pm \sqrt{1 - |\xi'|^2}$. Recalling that the symbol  of $h^2 \Deltab$ is $|\xi'|^2$, $\xi_n$ is, at least heuristically, the symbol of the operator $(1 - h_j^2 \Delta_{\dOmega})_+^{1/2}$,
where $t^{1/2}_+:=\sqrt{\max[t,0]}$. 

%
%
%
 Thus, we can conjecture that we might get better estimates on the quantity
$$
\big\| (1 - h_j^{2} \Delta_{\dOmega} )_+^{1/2} u^N_j \big\|_{L^2(\dOmega)}.
$$

This conjecture can be tested on the unit disc $D \subset \RR^2$. On the unit disc, 
given in polar coordinates by $0 \leq \theta < 2\pi$, $0 \leq r \leq 1$, Neumann eigenfunctions take the form 
$$
v_{n, l} = (A \cos n\theta  + B \sin n\theta) J_n(\mu_{n,l} r) 
$$
(apart from constants, associated to eigenvalue $0$) where $J_n$ is the Bessel function of order $n$ and $\mu_{n,l}$ is the $l$th positive zero of $J_n'(r)$. The eigenvalue is $E = \mu_{n,l}^2 = h_{n,l}^{-2}$; it is known that every eigenvalue (except 0) has multiplicity 2. We perform the following computation, using the well-known commutation identity $[\Delta, r \partial_r ] = 2\Delta$. 
\begin{equation}\begin{gathered}
2 E \| v \|_{L^2(D)}^2 = \int_{D} [\Delta - E, r \partial_r] v \cdot v \\
= \int_D (\Delta - E) (r \partial_r v) - (r \partial_r v) (\Delta - E) v \\
= \int_{\partial D}  (-\partial_r^2 v) v = \int_{\partial D} (E + \partial_\theta^2) v \cdot v \\
= (E - n^2) \| v \|_{L^2(\partial D)}^2. 
\end{gathered}\end{equation}
This shows that 
\begin{equation}
\frac{ \| v_{n,l} \|_{L^2(\partial D)} } {\| v_{n,l} \|_{L^2(D)} } = \sqrt{2} \big(1 - h_{n,l}^2 n^2 \big)^{-1/2}.
\label{v-ratio}\end{equation}
It is standard that the positive zeroes  of $J_n'$ start from $\mu_{n,1} = n + a n^{1/3} + O(n^{-1/3})$ with $a > 0$, and tend to infinity.  Since $h_{n,l} = \mu_{n,l}^{-1}$, this shows that the ratio \eqref{v-ratio} can be as large as $C h_{n,l}^{-1/3}$ --- illustrated by the whispering gallery mode shown in Fig.~\ref{f:intro}(a), where the eigenfunction is concentrated in a $h^{2/3}$ boundary layer ---
and as small as $\sqrt{2}$.

However,  at the boundary of the unit disc, $(1 - h_{n,l}^{2} \Delta_{\partial D})_+^{1/2} v_{n,l} = (1 - h_{n,l}^2 n^2 )^{1/2} v_{n,l}$. This yields  
\begin{equation}
\frac{ \| (1 - h_{n,l}^{2} \Delta_{\partial D})_+^{1/2} v_{n,l} \|_{L^2(\partial D)} } {\| v_{n,l} \|_{L^2(D)} } = \sqrt{2}.
\label{v-ratio-2}\end{equation}
Thus, we get uniform estimates (and even an exact identity) in the case of the unit disc. 

\subsection{Main results}

The discussion above motivates the question whether for an arbitrary smooth domain $\Omega$, there exist domain-dependent positive
constants $c$ and $C$ such that 
\begin{equation}
c \| u^N_j \|_{L^2(\Omega)} \leq  \| (1 - h_j^{2} \Delta_{\dOmega})_+^{1/2} u^N_j \|_{L^2(\partial \Omega)}  \leq C \| u^N_j \|_{L^2(\Omega)}.
\label{v-ratio-Omega}\end{equation}

We shall give a positive answer to a slightly different question with a modified spectral cutoff, in which the spectral weight factor $(1-h^2 \Deltab)^{1/2}$ is truncated at the level $h^{1/3}$.


Let $\chi_1, \chi_2 : \RR \to [0,1]$ be smooth functions such that $\chi_1(t) = 0$ for $t \leq 1$ and $\chi_1(t) = 1$ for $t \geq 2$, and such that $\chi_1^2 + \chi_2^2 = 1$. 
Consider the family of smooth functions $G_h$, depending on a parameter $h > 0$, given by 
\begin{equation}\begin{gathered}
G_h(\sigma) :=  \sqrt{\sigma} \chi_1 \big(\frac{\sigma}{h^{2/3}} \big) + h^{1/3}\chi_2 \big(\frac{\sigma}{h^{2/3}} \big) .
\end{gathered}\label{Gh}\end{equation}
This function regularizes the square-root in a smooth manner,
as sketched in Fig.~\ref{f:intro}(c). Let us motivate this choice of regularization. Since, as discussed above, heuristically an eigenfunction $u$, satisfying $(h^2 \Delta - 1) u = 0$, is composed of a superposition of oscillations $e^{i x \cdot \xi/h}$ with $|\xi| = 1$, the boundary values are composed of oscillations $e^{ix' \cdot \xi'/h}$, with $0 \leq |\xi'| \leq 1$. That is, the (semiclassically scaled) frequencies comprising the boundary values of eigenfunctions lie in the unit ball, $\{ |\xi'| \leq 1 \}$ of $T^* \dOmega$ --- the `classically allowed region'. This corresponds to the region where the symbol of $1 - h^2\Deltab$ is nonnegative. Therefore, the size of 
$\| G_h(1 - h^2 \Deltab) \ud \|_{L^2(\dOmega)}$ should be  rather insensitive to the behaviour of $G_h(\sigma)$ for $\sigma$ negative, allowing flexibility in the functional form of $G_h(\sigma)$ for $\sigma < 0$. In numerical applications we want to use the reciprocal, $F_h(\sigma) := 1/G_h(\sigma)$  as a spectral weight --- see \eqref{tttm-def}. Thus we want to choose $G_h(\sigma)$ strictly positive, and not too small, when $\sigma$ is negative. 
Regularization at the scale $h^{1/3}$ is then very natural, as it is the largest quantity that cancels the maximal possible growth of the $L^2$ norm of $\ud$ in view of \eqref{eq:Neumannbound}. We plot the resulting weight 
$G_h(1-|\xi'|^2)$ in terms of the scaled transverse frequency $\xi'$ in
Fig.~\ref{f:intro}(d).

Our modified question is then whether for any arbitrary smooth domain $\Omega$, there exist constants positive $c$ and $C$ such that 
\begin{equation}
c \| u_j \|_{L^2(\dOmega)} \leq  \| G_{h}(1 - h^2 \Delta_{\dOmega}) u_j \|_{L^2(\dOmega)}  \leq C \| u_j \|_{L^2(\dOmega)}, \quad |\sqrt{E_j} - h^{-1} | \leq 1. 
\label{v-ratio-Omega-2}\end{equation}

Our first theorem is a positive answer to this modified question. 

\begin{theo}\label{Neu-est}
Let $\Omega \subset \RR^n$ be a smooth bounded domain, and let $\ud_j$ be the restriction to $\partial \Omega$ of the $j$th $L^2$-normalized Neumann eigenfunction $u_j$. Then there are positive constants $c, C$ such that 
\begin{equation}
c \leq \|  G_{h}(1-h^2 \Deltab) \ud_j \|_{L^2(\partial \Omega)} \leq C,
\qquad
|h^{-1} -\sqrt{E_j}| \leq 1.
\label{cC}\end{equation}
\end{theo}

\begin{remark} The upper bound holds also for spectral clusters, that is, linear combinations of eigenfunctions with frequencies $h_j^{-1}$ in the range $[h^{-1}, h^{-1} + c]$,  and the lower bound also holds for clusters provided that $c$ is sufficiently small. In fact, the upper bound is more generally true for quasimodes satisfying $\| (h^2 \Delta - 1) u \|_{L^2(\Omega)} = O(h)$ and $\|  \nabla (h^2 \Delta - 1) u \|_{L^2(\Omega)} = O(1)$; see Proposition~\ref{prop:upperbound}. 
\end{remark}

\begin{remark} This is closely related to results of Tataru from the 1990s on the wave equation. In particular, the upper bound follows from \cite{Tataru}. See Section~\ref{subsec:lit} for more details. 
\end{remark}

We also prove a stronger version,  in which the upper bound is replaced by an upper bound for the norm of an operator built from all the eigenfunctions with eigenfrequency lying in an $O(1)$ frequency window. 
%

\begin{theo}[Quasi-orthogonality of spectrally weighted boundary values of Neumann eigenfunctions in a spectral window]\label{thm:qo}
The operator norm of \begin{equation}
\sum_{|\sqrt{E_j} - h^{-1}| \leq 1} G_{h}(1-h^2 \Deltab) \ud_j \ \Big\langle G_{h}(1-h^2 \Deltab) \ud_j, \ \cdot \ \Big\rangle 
\label{Nopnorm}\end{equation}
on $L^2(\dOmega)$ is uniformly bounded as $h \to 0$.
\end{theo}

\begin{remark}
By Weyl's law there are $O(E^{(n-1)/2})$ terms in the above sum,
and since by Theorem~\ref{Neu-est} each term has norm bounded below by
a constant, then Theorem~\ref{thm:qo} implies essentially complete
mutual orthogonality of the functions $G_{h}(1-h^2 \Deltab) \ud_j$,
up to a constant.
This is related to the pairwise quasi-orthogonality result
due to the first author,
\begin{equation}
\int_\pO (x\cdot n)\biggl( \frac{E_i+E_j}{2} u_i u_j - (d_t u_i)(d_t u_j)\biggr)
\; \le \; C (E_i-E_j)^2,
\label{pqo}\end{equation}
where $C= \sup_{x\in\Omega} \|x\|^2 /4$ is given explicitly.
Eq.~\eqref{pqo}
follows, e.g., by applying Neumann boundary conditions to the identity
\cite[Eq.~(C.6)]{que}. Indeed, the inner product
$$
\big\langle G_{h}(1-h^2 \Deltab) \ud_i , G_{h}(1-h^2 \Deltab) \ud_j  \big\rangle
$$
is analogous to \eqref{pqo} (up to the $x \cdot n$ factor), if we treat $G_h$ as a square-root, and integrate
by parts. 
Estimates \eqref{Nopnorm}, \eqref{pqo} are analogous to the quasi-orthogonality results for normal
derivatives of Dirichlet eigenfunctions proved in \cite{que,bnds}.
\end{remark}

Our second main result is an inclusion bound for Neumann eigenvalues. To state this, we define the modified tension appropriate to the Neumann boundary condition. Let $u$ be a smooth function on $\Omega$, which is continuous on the closure of $\Omega$.  We define
the inverse spectral weight $F_h(\sigma) = 1/G_h(\sigma)$,
then define a new tension $\tttm[u]$ by the quotient
\begin{equation}
\tttm[u] \;:=\; \frac{ \| F_h(1-h^2 \Deltab)(d_n u) \|_{L^2(\dOmega)}}{ \| u \|_{L^2(\Omega)}}.
\label{tttm-def}
\end{equation}

Notice that in this definition, the spectral multiplier $F_h$ acts, at least heuristically, as the inverse of the symbol of the normal derivative. So we can expect that  this tension is a good analogue of the Dirichlet tension \eqref{t-def} in the high energy limit. 
Our second theorem justifies this heuristic, by proving a Neumann inclusion bound in which both the upper and lower bounds are uniform as $h \to 0$.

\begin{theo}\label{MPS-Neu} There exist positive constants $c, C$ depending only on $\Omega$ such that the following holds.  Let $u$ be any  nonzero solution of $(\Delta - E) u = 0$ in $C^\infty(\Omega)$. Let $\tttm[u]$ be as in \eqref{tttm-def} with $h^{-2} = E$, and let $\umin$ be the Helmholtz solution minimizing $\tttm[u]$. Then
\begin{equation}
c   \tttm[\umin] \leq d(E, \spec^N) \leq C \tttm[u].
\label{inclusion-eval}\end{equation}
Moreover, if $J$ is such that $E_J$ is the closest Neumann eigenvalue to $E$, $E_* \neq E_J$ is the next closest Neumann
eigenvalue, and $\Pi_J$ is orthogonal projection onto the $E_J$-eigenspace, then there is a constant $C$ such that 
\begin{equation}
\frac{ \| u - \Pi_J u \|_{L^2(\Omega)}}{\| u \|_{L^2(\Omega)}} \leq C \frac{ \tttm[u]}{|E - E_*|}.
\label{inclusion-efn}\end{equation}
\end{theo}

\begin{remark}
Existing classical Neumann inclusion bounds, e.g.\
taking the case of a Helmholtz trial function in
\cite[Eq.~(14)]{KSbounds} or \cite[Thm.~7]{EnNeu},
bound the error
in the eigenvalue by a constant times $E \|d_n u\|_{L^2(\pO)} / \|u\|_{L^2(\Omega)}$.
Since the operator norm $\|F_h(1-h^2\Deltab)\|_{L^2(\pO)} = O(h^{-1/3}) = O(E^{1/6})$,
this means that our new bound is 
a factor at least $E^{5/6}$ tighter than
existing bounds at high frequencies.
For comparison, the factor of improvement in the Dirichlet inclusion bound in \cite{bnds} was $E^{1/2}$.
\label{r:factor}
\end{remark}

Theorem~\ref{MPS-Neu} is illustrated by Fig.~\ref{f:tsweep},
which plots a good numerical
approximation to $\tttm[\umin]$ vs $E$ over two different $E$ ranges.
Eq. \eqref{inclusion-eval} implies that,
in the neighbourhood of each eigenvalue,
such graphs must lie between `absolute value' type functions with
slopes uniformly bounded above and below
(also see Remark~\ref{melissa}).
Indeed, it is visually clear
that the slopes are very similar and independent of $E$.


\subsection{Analytic techniques}
A feature of this paper is that we need only relatively standard analytic tools to obtain our results; in particular we
use no microlocal analysis. We make use of 
\begin{itemize}
\item Energy estimates, of a sort familiar in hyperbolic equations, but here adapted to the elliptic setting.
\item The Helffer-Sj\"ostrand formula for functions of self-adjoint operators --- see \eqref{HSform}. This is our main technical tool for analyzing functions of operators. 
\item Standard elliptic PDE estimates of the sort found in Gilbarg-Trudinger \cite{MR1814364}, Taylor \cite{TaylorI} or Evans \cite{evanspde}, which tell us 
that elliptic operators have good mapping properties between Sobolev spaces. More precisely, let $M$ be a compact manifold (such as $\dOmega$), and let $P$ be an elliptic operator of second order on $M$. Then if $Pu = f$, and if $f \in H^s(M)$, $s \geq 0$, then 
\begin{equation}
\| u \|_{H^{s+2}(M)} \leq C \Big( \| f \|_{H^{s}(M)} + \| u \|_{L^2(M)} \Big).
\end{equation}
(This is readily deduced from the corresponding `interior' estimate on open subsets of $\RR^n$.) Moreover, if $P$ is invertible, then we can remove the $L^2$ norm of $u$ on the RHS. We need just a minor semiclassical variant of this estimate:  we use semiclassical Sobolev spaces $H^s_h(M)$, $s \in \NN$,  with norms defined by 
$$
\| u \|_{H^s_h(M)}^2 = \sum_{|\alpha| \leq s} \| (h D)^\alpha u \|_{L^2(M)}^2;
$$
these norms are equivalent to the usual norm for each fixed $h$, but not uniformly as $h \to 0$. Then if $P$ is a semiclassical elliptic operator, and $Pu = f$, we get estimates in these spaces 
\begin{equation}
\| u \|_{H^{s+2}_h(M)} \leq C \Big( \| f \|_{H^{s}_h(M)} + \| u \|_{L^2(M)} \Big)
\label{sclSob}\end{equation}
with $C$ independent of $h$ \cite[Theorem 7.1]{zworski}. Moreover, if $P$ is invertible in a uniform sense, for example if $P \geq 1$, then we can remove the $u$ term on the RHS. 
\end{itemize}

\subsection{Related eigenfunction literature}\label{subsec:lit} 

Estimates on the concentration of eigenfunctions of the Laplacian and related operators have a rich literature, see for example  surveys by Zelditch \cite{Z09} or Sogge \cite{sogge15}, and the books \cite{MR1205579}, \cite{MR3186367} by Sogge. As the literature is quite large, we mention only a few of the key related papers below. 

The `modern history' of eigenfunction estimates on manifolds begins with Sogge's classic 1988 paper \cite{sogge88}, in which he obtained sharp $L^p$ estimates for eigenfunctions on compact manifolds for all $p \in [2, \infty)$. The key approach was to estimate the operator norm, acting from $L^2$ to $L^p$, of a smoothed spectral projector. These estimates hold for spectral clusters (combinations of eigenfunctions within a frequency interval of unit length) just as well as they do for eigenfunctions. The analogous estimates on manifolds with boundary have been studied, but are complicated by phenomena such as `whispering gallery modes', which prevent such results holding in generality, for example on convex domains. However in \cite{SS95} Smith and Sogge proved that if boundary is everywhere strictly geodesically concave the results from boundaryless manifolds hold. In the general case, where the boundary may be convex, Smith and Sogge \cite{SS07} proved a sharp set of bounds for dimension two. They were also able to obtain results in higher dimensions but it is not known whether those results are sharp. In \cite{Blair12} Blair obtained estimates for the restriction of eigenfunctions to hypersurfaces in the low regularity setting (which includes the boundary case). Extreme concentration of eigenfunctions on manifolds was studied in \cite{soggegrowth, SZ2}.    See also the related work on  Strichartz estimates on manifolds with boundary by Blair, Smith and Sogge \cite{BSS06,BSS08}.

Estimates on the boundary values of eigenfunctions, or quasimodes (approximate eigenfunctions) have been obtained as a by-product of estimates on solutions of the wave equation. For example, estimates on the normal derivatives of Dirichlet eigenfunctions can be read off from the main result of \cite{BLR}. The case of Neumann boundary conditions was studied by Tataru in \cite{Tataru}, \cite{Tataru-1998}. As there is considerable overlap between these results and those of the present paper, we describe them in more detail. 
In \cite{Tataru}, Tataru showed that a certain norm on the boundary values of a solution to the wave equation is bounded by the interior $H^1$ norm. Specializing to time-harmonic solutions, and using our notation, it shows that the norm 
\begin{equation}
h^{1/3} \| u \|_{L^2(\dOmega)} + \| (1 - h^2 \Deltab)^{1/2}_+ u \|_{L^2(\dOmega)} +  \| (1 - h^2 \Deltab)^{1/2}_- u \|_{L^2(\dOmega)}
\label{Tataru-norm}\end{equation}
 is bounded by the interior $L^2$ norm. In \cite{Tataru-1998}, an argument is sketched to show that this is in fact equivalent to the interior $L^2$ norm. A priori, \eqref{Tataru-norm} is larger than our $G_h(\Deltab)$ norm, so Tataru's result is stronger than our Theorem~\ref{Neu-est} for the upper bound but  a little weaker for the lower bound.\footnote{Since both norms are equivalent (on eigenfunctions) to the interior $L^2$ norm, they are equivalent to each other, but this is not a priori clear. The direct equivalence of these two boundary norms can be deduced from our exterior mass estimate in Section~\ref{sec:eme}.} We note that the method of proof is quite different: Tataru's proof uses rather sophisticated microlocal analysis, while ours uses no microlocal analysis at all. 
 
  Estimates on normal derivatives of Dirichlet eigenfunctions were also obtained in \cite{HT}, where the energy method used here in Section~\ref{sec:1/3-bound} was used. Ergodic properties of boundary values of eigenfunctions were studied in \cite{gerard93}, \cite{hassell}, \cite{Bu}. 

Another strand of current research is to estimate eigenfunctions restricted to interior submanifolds, which was studied in 
\cite{BGT07}, \cite{tacy09}, \cite{Blair12}, \cite{MR3340369} and in the ergodic  case in \cite{MR2913617}, \cite{MR3053760}.

A slightly different framework for obtaining estimates on quasimodes of more general semiclassical operators was introduced by  Koch, Tataru and Zworski \cite{koch07}. These results have been extended in \cite{tacy09}, \cite{MR2868958}.

\subsection{Plan of this paper}

The plan of the rest of this paper is as follows. In Section~\ref{sec:1/3-bound} we prove a bound $C h_j^{-1/3}$ for the $L^2$ norm of $u^N_j$ restricted to $\dOmega$. This result is implicit in Tataru's work on boundary traces of solutions of the wave equation \cite{Tataru} mentioned above, but as the result is not particularly accessible in the literature we give a direct proof here based on energy estimates, following the paper \cite{HT}. In Section~\ref{sec:modE} we introduce a modified energy and prove an upper bound on the `$\chi_1$' part (see \eqref{Gh}) of $G_h(1-h^2 \Delta_{\dOmega})u^N_j$, that is, the part in the classically allowed region.  
Using the modified energy and the $h^{-1/3}$ upper bound we prove the upper bound in Theorem~\ref{Neu-est}. 
In Section~\ref{sec:eme} we prove an estimate on the exterior mass, that is, on the  part of $G_h(1-h^2 \Delta_{\dOmega})u^N_j$ that is  sufficiently far (at the scale $h^{2/3}$) from the classically allowed region. In Section~\ref{sec:lowerbound}, we use the exterior mass estimate to prove the lower bound from Theorem~\ref{Neu-est}. In Sections~\ref{sec:lowerbound}--\ref{sec:inclusion} we prove Theorem~\ref{thm:qo}, and finally the Neumann inclusion bound, Theorem~\ref{MPS-Neu}. 

In the last section, we present an improved implementation of the MPS for the 
high-accuracy computation of Neumann eigenvalues of smooth planar domains,
which exploits our tight inclusion bounds for the tension $\tilde{t}_h$. 
In one example our theorem leads to error bounds that are improved by 5 digits over the existing Neumann MPS 
with little extra numerical effort.


\section{The $h^{-1/3}$ bound on boundary values of Neumann eigenfunctions}\label{sec:1/3-bound}

In this section, we establish Tataru's $h^{-1/3}$ estimate for the $L^2$ norm of Neumann or Robin eigenfunctions restricted to the boundary, which is just the upper bound for the first term of \eqref{Tataru-norm}. We adapt the method from Hassell--Tao \cite{HT}. Since the boundary of $\Omega$ is assumed to be smooth, there is a collar neighbourhood on which we have Fermi coordinates $r$ and $y$, where $r \in [0, \delta]$ is distance to the boundary, a smooth function for $\delta > 0$ sufficiently small, and $y$ is constant on geodesics normal to the boundary. We denote such a collar by $\Cd$, and we denote the set of points at distance $r$ from the boundary by $Y_r$; this is a smooth submanifold for $r \leq \delta$. The advantage of Fermi coordinates is that, in these coordinates the vector fields $\partial_r$ and $\partial_{y_j}$ are orthogonal; thus the metric takes the form 
$$
g = dr^2 +  k_{ij}(r) dy_j dy_j
$$
where we sum (from $1$ to $n-1$) over repeated indices. 

\begin{remark} In this section, and the next, $\Omega$ can be any compact Riemannian manifold with smooth boundary. Indeed, all the upper bound results in this paper hold in this generality. However, for the lower bound to hold, it is known in the Dirichlet case that one must have a `nontrapping' or `geometric control' condition, that is, that all billiard trajectories of some given length $T$ meet the boundary nontangentially --- see \cite{BLR} and \cite{HT}. We intend to explore this further in a future article. 
\end{remark}

To analyze the Laplacian, it is convenient to make a unitary transformation $U$  from $L^2(\Cd)$, with the Riemannian measure, to $L^2([0, \delta] \times \dOmega)$ with the product measure $dr d\gd$, where $\gd$ is the induced metric on $\dOmega$. This is done by multiplication with a suitable factor,
$$
L^2(\Cd) \ni u \mapsto v  = Uu := u \gamma, \quad \gamma(r, y) = \Big( \frac{ \det k_{ij}(r, y)}{\det k_{ij}(0, y)} \Big)^{1/4}.
$$
Here, $\gamma$ is a well-defined function on $\Cd$, independent of the choice of local coordinates $y$. 
Then a standard calculation shows that the conjugation 
$U (h^2 \Delta - 1) U^{-1}$ takes the form
\begin{equation}
U (h^2 \Delta - 1) U^{-1} = -(h^2\partial_r^2 + \tilde P(r)),
\label{tconj}\end{equation}
where $\tilde P(r)$ is a self-adjoint, elliptic semiclassical differential operator on $\dOmega$, with coefficients depending on $r$. In fact, it takes the form in local coordinates 
$$
\tilde P(r) = 1+ h^2  A^{-1} \partial_{y_i} \Big( k^{ij}(r) \partial_{y_j} A \, \cdot \Big)+ h^2 b  , \quad A = (\det k_{ij}(0,y))^{1/4}, \quad b \in C^\infty(\Cd). 
$$
Here $k^{ij}$ is the inverse matrix of $k_{ij}$. Moreover, we claim that
\begin{equation}
\tilde P(0) =  1-h^2 \Delta_{\dOmega} + h^2 b',
\label{tP0}\end{equation}
where $b'$ is a smooth function. 
To see why \eqref{tP0} is true, we first conjugate by $(\det k_{ij}(r, y))^{1/4}$, and afterward by $(\det k_{ij}(0, y))^{-1/4}$.
It is well known that the first conjugation reduces the Laplacian to divergence form, plus a zeroth order term, i.e. the form
$$
 -\partial_r^2 - \partial_{y_i} \Big( k^{ij}(r, y) \partial_{y_j} \, \cdot \Big)  + b .
$$
Now consider the conjugation by $(\det k_{ij}(0, y))^{1/4}$. This term commutes with $\partial_r$, so it produces an operator of the form 
\begin{equation}
-\partial_r^2 + \text{ (tangential derivatives of order $1$ and $2$) } + b'. 
\label{Pform}\end{equation}
In particular, there are no first order $r$-derivatives. On the other hand, the factor $\gamma$ is equal to $1 + O(r)$ at the boundary. 
At the boundary, the Laplacian in Fermi coordinates takes the form
$$
-\partial_r^2 - a \partial_r + \Delta_{\dOmega}.
$$
Conjugating with $1 + O(r)$ changes this, at $r=0$, to 
$$
-\partial_r^2 - a' \partial_r + \Delta_{\dOmega} + b'.
$$
Here $a, a', b'$ are smooth functions on $M$. 
Comparing this to the form \eqref{Pform}, we see that  $a' = 0$. This establishes \eqref{tP0}. (Moreover, we could identify $b'$ explicitly in terms of the mean and scalar curvatures of $\partial \Omega$ --- see \cite[end of proof of Theorem 1]{MR3396090} for a similar calculation.) 

We now define
\begin{equation}
P(r) = \tilde P(r) - h^2 b',
\label{Prdef}\end{equation}
so that 
\begin{equation}
P(0) =  1-h^2 \Delta_{\dOmega} 
\label{P0}\end{equation}
and 
\begin{equation}
U (h^2 \Delta - 1) U^{-1} = -(h^2\partial_r^2 + P(r) + h^2 b').
\label{conj}\end{equation}

Now suppose that $u$ is a real-valued approximate eigenfunction in the sense that 
\begin{equation}\begin{gathered}
(h^2\Delta - 1) u = w, \quad \| u \|_{L^2(\Omega)} = 1, \\
\| w \|_{L^2(\Omega)} \leq c h, \\ 
d_r u = 0  \text{ at } \dOmega.  
\end{gathered}\label{approxefn}\end{equation}

We let $v = Uu$. Then on $[0, \delta] \times \dOmega$,  $v$ satisfies the equation
\begin{equation}
-(h^2 \partial_r^2 + P(r))v = w' := Uw +  h^2 b' Uu, \quad \| w' \|_{L^2([0, \delta] \times \dOmega)} \leq c h.
\label{veqn}\end{equation}
In preparation for our $h^{-1/3}$ estimate on the boundary values of Neumann eigenfunctions, we prove an estimate on the following energy functional introduced (apart from some inessential changes) in \cite{HT}. For each $r \in [0, \delta]$, and for $v = Uu$ as above, we define the energy $E(r)$ given by
\begin{equation}
E(r) = h^2 \ang{v_r, v_r}_{\{ r \} \times \dOmega} + \ang{Pv, v}_{\{ r \} \times\dOmega} = \int_{Y_r} \big(h^2 v_r^2 + (Pv) v \big) \, d\gd . 
\label{E1}\end{equation}
In \eqref{E1}, we consider the functions $v_r$ and $v$ as defined on the product $[0, \delta] \times \dOmega$, with the product measure $dr d\gd$. That is, the inner product in \eqref{E1} is given by that on $\dOmega$, through our identification of $Y_r$ and $\dOmega$ along geodesics normal to the boundary. The inner product is 
thus fixed, independent of $r$, which makes calculating the $r$-derivative of $E(r)$ more straightforward. Notice that $E(r)$ does not have a fixed sign, as the operator $P$ has spectrum in $(-\infty, 1]$. 

\begin{remark} The energy above is a formal analogue, in the elliptic setting, for the well-known energy for the wave equation, which is 
$
\int \big( |u_t(x, t)|^2 + |\nabla u(x, t)|^2 \big) \, dx
$
integrated over a time slice. 
Similar energies have been used elsewhere, for example \cite[Eq. (3.14)]{Katogrowth}, \cite[below Eq. (2.9)]{Datchev}, etc.
\end{remark}

\begin{lem}\label{lem:Energy} Suppose that $u \in C^2(\overline{\Omega})$ and that $u$ satisfies the Neumann boundary condition $d_n u = 0$ at $\dOmega$. Then there is a constant $C$ depending only on $\Omega$ such that for all $h \leq 1$, the energy \eqref{E1} for $v = Uu$ satisfies 
\begin{equation}|E(r)| \leq C \Big( \| u \|_{L^2(\Omega)}^2 + h^{-2} \| (h^2 \Delta - 1) u \|_{L^2(\Omega)}^2 \Big), \quad r \in [0, 2\delta/3].
\label{Ebound}\end{equation}
\end{lem}

\begin{proof}
To prove the lemma, we choose a smooth function $\eta$ such that $\eta = 1$ on $[0, 2\delta/3]$ and $\eta(r) = 0$ for $r \geq \delta$. Then we have
\begin{equation}\begin{gathered}
|E(s)| \leq \int_s^\delta \big| \frac{d}{dr} \eta(r) E(r) \big| \, dr \\
\leq \int_0^\delta \Big( |E(r)| + \big| \frac{dE}{dr} \big| \Big)  \, dr.
\end{gathered}\label{E2}\end{equation}
To estimate $\int |E(r)|$, we note that $P(r)$ only involves tangential derivatives, so we can integrate by parts once and estimate
$$
 |E(r)|   \leq C\int_{Y_r} \big( v^2 + h^2|\nabla v|^2 \big) \leq C' \int_{Y_r} \big( u^2 + h^2|\nabla u|^2 \big) . 
$$

To estimate the derivative term $|E'(r)|$, we compute, as in \cite{HT}, with the dot indicating derivative with respect to $r$ 
\begin{equation}\begin{gathered}
\frac{dE}{dr} = 2h^2 \ang{v_{rr}, v_r} + 2 \ang{P(r) v, v_r} + \ang{\dot P(r) v, v}. 
\end{gathered}\label{E3}\end{equation}
(Notice that, since we transformed our measure to be independent of $r$, there is no derivative term due to the change of measure.) Using \eqref{veqn}, we find that
\begin{equation}\begin{gathered}
\frac{dE}{dr} = -2 \ang{w', v_r} + \ang{\dot P(r) v, v }, \quad w' = U \big( (h^2 \Delta - 1) u + h^2 b' u\big). 
\end{gathered}\label{E3.5}\end{equation}
We treat the $\dot P(r)$ term as with the $\ang{P(r) v, v}$ term above, and find that 
\begin{equation}
|E(r)| + |E'(r)| \leq C\int_{Y_r} \big( u^2 + h^2|\nabla u|^2 - h^{-1} (w + h^2 b' u) (h u_r) \big). 
\label{E+E'}\end{equation}
Substituting into \eqref{E2} we find that 
\begin{equation}
|E(s)| \leq C\int_{\Cd} \big( u^2 + h^2|\nabla u|^2 - h^{-1} (w + h^2 b' u) (h u_r) \big), \quad s \in [0, 2\delta/3]. 
\label{E4}\end{equation}
Now we estimate 
\begin{equation}
h^2 \int_{\Omega} |\nabla u|^2 =  \int_{\Omega} u ( h^2 \Delta u ) = \int_{\Omega} u (u + w) \leq  \| u \|_{L^2(\Omega)}^2 +  \| u \|_{L^2(\Omega)} \| w \|_{L^2(\Omega)},
\end{equation}
where $w = (h^2 \Delta - 1) u$, 
which implies that 
\begin{equation}
h^2 \int_{\Cd} |\nabla u|^2 \leq C \Big( \| u \|_{L^2(\Omega)}^2 + h^{-2} \| w \|_{L^2(\Omega)}^2 \Big).
\label{nablau}\end{equation}
Combining \eqref{E4} and \eqref{nablau} proves the lemma. 
\end{proof}

\begin{prop}\label{1/3}
Suppose that $u \in C^2(\overline{\Omega})$ and that $u$ satisfies the Neumann boundary condition $d_n u = 0$ at $\dOmega$. Then there is a constant $C$ depending only on $\Omega$ such that for all $h \leq 1$, and all $r \in [0, \delta/3]$, we have 
\begin{equation}
\| u \|_{L^2(Y_r)}^2 \leq C h^{-2/3}\Big( \| u \|_{L^2(\Omega)}^2 + h^{-2} \| (h^2 \Delta - 1) u \|_{L^2(\Omega)}^2 + \| \nabla (h^2 \Delta - 1) u \|_{L^2(\Omega)}^2 \Big) .
\label{Tataru-bound}\end{equation}
In particular, for normalized Neumann eigenfunctions $u$, with eigenvalue $h^{-2}$, we have
\begin{equation}
\| \ud \|_{L^2(\dOmega)} \leq C h^{-1/3}. 
\label{eq:Neumannbound}\end{equation}
\end{prop}

\begin{remark} The exponent $1/3$ sharp as shown by the example of the disc. 
\end{remark}

\begin{remark} This result is essentially contained in Tataru \cite{Tataru}. However, \eqref{Tataru-bound} and \eqref{Tataru-bound-improv} are slight variants which will be convenient in Section~\ref{sec:inclusion}. 
\end{remark}

\begin{remark} This result can be generalized in various ways. For example, we can prove the same result for boundary conditions of the form $hd_n u= b u$, where $b \geq -c_2$. Then we obtain the same estimate with constant $C$ now depending on $c_2$ as well as $c_1$ and $\Omega$. 

Another straightforward generalization is that the operator $h^2 \Delta - 1$ can be a more general semiclassical operator, e.g. $h^2 \Delta + V(x) - 1$.
\end{remark}

\begin{proof}
For notational convenience, we scale $u$ so that 
$$
 \| u \|_{L^2(\Omega)} + h^{-1} \| (h^2 \Delta - 1) u \|_{L^2(\Omega)} + \| \nabla (h^2 \Delta - 1) u \|_{L^2(\Omega)}  = 1.
 $$
 Then we need to show that $\| u \|_{L^2(Y_r)} \leq C h^{-1/3}$. 

As in the proof of Lemma~\ref{lem:Energy}, we switch to $v= Uu$. 
Following \cite{HT}, we define
$$
L(r) = h^2 \int_{Y_r} v^2 \, dy. 
$$
We can compute, with dots indicating derivatives in $r$, 
\begin{equation}
\dot L(r) = 2h^2 \int_{Y_r}  v v_r \, dy
\label{L'}\end{equation}
and 
\begin{equation}\begin{gathered}
\ddot L(r) = 2h^2  \int_{Y_r} (v_r^2 + v v_{rr} ) \, dy \\
= 2 \int_{Y_r} 4h^2 v_r^2 - 2 E(r) \pm 2 w' v \\
\geq \frac{\dot L(r)^2}{L(r)} - 2 E(r) - 2 C' \\
\geq \frac{\dot L(r)^2}{L(r)} - C. 
\end{gathered}\label{L-feedback}\end{equation}
Here, we used equation \eqref{veqn} for $v$ in the second line, Cauchy-Schwarz (applied to the integral in \eqref{L'}) in the third line, and \eqref{Ebound}  in the fourth line. Also, in the third line, we used the estimates 
\begin{equation}\begin{aligned}
\| w' \|_{L^2(Y_r)}^2 &\leq \| w' \|_{L^2(\Omega)} \| \nabla w' \|_{L^2(\Omega)} + C \| w' \|_{L^2(\Omega)}^2 \\
\| v \|_{L^2(Y_r)}^2 &\leq \| v \|_{L^2(\Omega)} \| \nabla v \|_{L^2(\Omega)} + C \| v \|_{L^2(\Omega)}^2 \\
\end{aligned}\end{equation}
(which is an immediate consequence of the fundamental theorem of calculus) to estimate $\| w' \|_{L^2(Y_r)} \leq C h^{1/2}$ and  $\| v \|_{L^2(Y_r)} \leq C h^{-1/2}$. Thus, we have 
$$
\int_{Y_r} w' v \geq - C.
$$

From this equation, we show, as in the erratum \cite{MR2661180} of Hassell--Tao,
that 
\begin{equation}
\dot L(r) > 0 \implies (\dot L(r))^2 < 4C L(r) \text{ for all } r \in [0, 2\delta/3].
\label{claim}\end{equation}
For if this is not so, then we have $\dot L(r_0) > 0$ and $(\dot L(r_0))^2 \geq 4C L(r_0)$ for some $r_0 \in [0, 2\delta/3]$. Plugging into the last line of \eqref{L-feedback}, we obtain $\ddot L(r_0) > 3C $, and this shows that
$$
\frac{d}{dr} \Big( (\dot L(r))^2 - 4C L(r) \Big) = 2 \dot L(r) \ddot L(r) - 4C \dot L(r) \geq 2C \dot L(r) > 0 \text{ for } r = r_0.
$$
This means that the properties $\dot L(r)$ positive, and $(\dot L(r))^2 \geq 4C L(r)$ persist in some interval $[r_0, r_0 + \epsilon)$, $\epsilon > 0$. By a continuity argument we see that these properties hold for all $r \in [r_0, \delta]$. But then we would have
$$
\frac{d}{dr} \sqrt{L(r)} \geq \sqrt{C}, \quad r \in [r_0, \delta]. 
$$
This implies 
$$
L(r) \geq C (r - r_0)^2, \quad r \in [r_0, \delta],
$$
which is a contradiction for sufficiently small $h$, since
\begin{equation}
\int_0^\delta L(r) \, dr \leq h^2 \int_\Omega u^2 \leq h^2.
\label{norm}\end{equation}
We deduce that \eqref{claim} holds. 

Next we improve \eqref{claim} by showing that 
\begin{equation}
(\dot L(r))^2 \leq 5C L(r) \text{ for all } r \in [0, 2\delta/3],
\label{claim2}\end{equation}
that is, we remove the assumption $\dot L(r) > 0$. 
Again we do this by contradiction. Assume \eqref{claim2} fails; then there exists 
$r^* \in (0, 2\delta/3]$ such that $\dot L(r^*)^2 > 5C L(r^*)$. Since
$L(r) = 0 \implies \dot L(r) = 0$, we see that $L(r^*) > 0$. By continuity, there is some interval $[r_*, r^*]$, $r_* < r^*$,  on which 
\begin{equation}
\text{
$\dot L(r)^2 \geq 5C L(r)$ and $L(r) \geq L(r^*)/2$ for all $r$ in $[r_*, r^*]$.}
\label{s**def}\end{equation}
 Let $r_{**}$ be the infimum of all such $r_*$. Notice that $r_{**} > 0$, as $\dot L(r_{**}) > 0$. 
We must have either (i) 
$$
L(r_{**}) = L(r^*)/2 \text{ and } \dot L(r_{**})^2 \geq 5C L(r_{**})
$$
or (ii) 
$$
L(r_{**}) > L(r^*)/2 \text{ and } \dot L(r_{**})^2 = 5C L(r_{**})
$$
In case (i), $\dot L(r_{**}) \geq 0$, since $L$ has its minimum value on the interval $[r_{**}, r^*]$ at $r_{**}$. From the second inequality we see that in fact $\dot L(r_{**}) > 0$. But this contradicts \eqref{claim}. So (i) is not possible. Next consider (ii). In this case, since $r_{**}$ is the infimum of $r_* \in (r_{**} - \epsilon, r^*) $ satisfying $\dot L(r_{*})^2 \geq 5C L(r_{*})$, we have 
$$
\frac{d}{dr} \Big( \dot L(r)^2 - 5C L(r) \Big) \geq 0 \text{ at } r= r_{**}.
$$
This can be rewritten
$$
2\dot L(r_{**}) \Big( \ddot L(r_{**}) - 5C/2 \Big) \geq 0.
$$
But from \eqref{L-feedback} we have $\ddot L(r_{**}) \geq \dot L(r_{**})^2/L(r_{**}) - C \geq 4C$, so this shows that
$\dot L(r_{**}) \geq 0$. As before, this implies that $\dot L(r_{**}) > 0$ and again contradicts \eqref{claim}. This  establishes \eqref{claim2}.

We are now in a position to prove that $L(r) \leq M C h^{4/3}$ for some $M$ and all  $r \leq \delta/3$ and sufficiently small $h$, which is equivalent to \eqref{Tataru-bound}. The idea is that if $L(r)$ is very big, then \eqref{claim2} shows that $L(r)$ will be remain big in some interval around $r$, which will lead to a contradiction with \eqref{norm}. 

Equation \eqref{claim2}  implies that, with a slightly bigger $C$, 
$$
\big| \frac{d}{ds} \sqrt{L(s)} \big| \leq \sqrt{C}  . 
$$
Therefore, for $r \in [0, \delta/3]$, $s \in [r, 2\delta/3]$, we have
$$
\sqrt{L(s)} \geq  \sqrt{L(r)} - \sqrt{C} (s-r)  .
$$
Now suppose that $L(r) \geq MC h^{4/3}$ with $M \geq 1$. Then we would have
$$
\sqrt{L(s)} \geq \sqrt{MC} (h^{2/3} - (s-r)) \geq \sqrt{MC} \frac{h^{2/3}}{2} \text{ for } s \in [r, r+h^{2/3}/2]
$$
provided $h$ is small enough so that $h^{2/3}/2 \leq \delta/3$. 
Squaring and integrating between $r$ and $r+h^{2/3}/2$, we find that
$$
\int_r^{r+h^{2/3}/2} L(s) \, ds \geq \int_0^{h^{2/3}/2} MC  \frac{h^{4/3}}{4} \, dr = \frac{MCh^2}{8} 
$$
which contradicts \eqref{norm} for $M > \max(1, 8/C)$.
\end{proof}

\begin{remark}\label{rem:w-improvement}
The sharp-eyed reader may notice that we can slightly strengthen the result by replacing the RHS of 
\eqref{Tataru-bound} by 
\begin{equation}\begin{gathered}
C h^{-2/3} \Big( \| u \|_{L^2(\Omega)}^2 + h^{-1} \| u \|_{L^2(\Omega)}^{1/2} \| w \|_{L^2(\Omega)}^{3/2}   \\
+ \ \big( \| u \|_{L^2(\Omega)} + \| u \|_{L^2(\Omega)}^{3/4} \| w \|_{L^2(\Omega)}^{1/4} \big)  \| w \|_{L^2(\Omega)}^{1/2}  \| \nabla w \|_{L^2(\Omega)}^{1/2}
\Big)
\end{gathered}\label{Tataru-bound-improv}\end{equation}
where $w = (h^2 \Delta - 1) u$. The key improvement here is that the $\nabla w$ term only appears with exponent $1/2$. This improvement plays a role in the proof of the inclusion bound in Section~\ref{sec:inclusion}. 
\end{remark}


%
%
%
%

\section{Modified energy}\label{sec:modE}
In this section, we define a modified energy, using a spectral cutoff in the $P(r)$ operator, where $P(r)$ is as in \eqref{Prdef} and \eqref{conj}. We then carry out the same computation as in Lemma~\ref{lem:Energy} to show that this modified energy is also uniformly bounded in $h$. 

Recall that the semiclassical operator $P(r)$, acting on functions defined on $Y_r$, is approximately given by $1-h^2 \Delta_r$, where $\Delta_r$ is the Laplacian on $Y_r$. We shall cut off spectrally in the region where $P(r)$ is positive. We want a fairly sharp cutoff; it turns out that the correct scale is $h^{2/3}$, i.e. the cutoff function we can transition from $0$ to $1$ in an interval of length $h^{2/3}$ in the spectrum of $P(r)$ and no faster, for the estimates below to be valid. 

Let $\chi : \RR \to [0,1]$ satisfy $\chi(t) = 0$ for $t \leq 1$ and $\chi(t) = 1$ for $t \geq 2$. Let 
$f(t)$ be the $h$-dependent function of $t$ given by 
\begin{equation}
f(t) = \chi \big( \frac{t}{h^{2/3}} \big). 
\label{fdefn}\end{equation}

\begin{defn}
Let $f$ be given by \eqref{fdefn}, and let $u$ be a function on $\Cd$. 
Then the modified energy $\tE$  is defined to be
\begin{equation}
\tE(r) = \ang{f(P)hv_r, hv_r} + \ang{f(P) P v, v}, \quad v = Uu
\end{equation}
where the inner products are on $L^2(Y_r)$ with fixed Riemannian measure $d\gd$. 
\end{defn}

\begin{remark}
Notice that $f(P)$ and $f(P) P$ are positive operators. Hence $\tE(r)$, unlike $E(r)$, is a positive quantity for each $r$. Curiously, we make no use of this fact in the present paper. 
\end{remark} 

To compute with the modified energy $\tE(r)$, we need to be able to take $r$-derivatives of $f(P) = f(P(r))$. The following lemma will be useful. 

\begin{lem}\label{lem:F}
The operator $(d/dr)^k f(P)$ is bounded in operator norm by $C_k h^{-2k/3}$ for $k = 0, 1, 2, 3$. 
\end{lem}

\begin{proof}
We select an almost analytic extension of $f$,  denoted $F$. We do this as follows: using H\"ormander's method we extend the smooth, compactly supported function 
$$
g(t) = \begin{cases} 
\chi(t), \quad t \leq h^{-2/3} \\
\chi(h^{-2/3} +2-t), \quad t \geq h^{-2/3}
\end{cases}$$
to an almost analytic extension $G(z)$, satisfying
$$
\Big| \partial_z^k \overline{\partial} G(z) \Big| \leq C_{k, N} |\Im z|^N \quad \forall \, z \in \CC, \quad \forall N \in \NN. 
$$
Notice that these estimates are uniform in $h$. We
 then define 
 $$
 F(z) = F_h(z) = G(z h^{-2/3}).
 $$
 Notice that $F(t) = f(t)$ for $t \leq 1$ real, and since $P(r) \leq 1$ as an operator, it follows that $F(P) = f(P)$. 
 
 Due to the scaling in the definition of $h$, we have 
 \begin{equation}
 \Big| \partial_z^k \overline{\partial} F(z) \Big| \leq C_{k, N} h^{-2(k+N+1)/3} |\Im z|^N \quad \forall \, z \in \CC, \quad \forall N \in \NN. 
 \label{F-est}\end{equation}
In addition, we can assume that $G$ is supported in the set $[0, h^{-2/3} + 2] \times i[-1, 1]$. Consequently, $F$ is supported in $[-2h^{2/3}, 1 + 2h^{2/3}] \times i[-h^{2/3}, h^{2/3}]$, which is a set of measure $O(h^{2/3})$.

 We can express $f(P)$ in terms of $F$ using the standard Helffer-Sj\"ostrand formula \cite[Theorem 8.1]{MR1735654}
 \begin{equation}
f(P) = \frac1{2\pi} \int_{\CC} \dbar F(z) (P - z)^{-1} \, dL(z). 
\label{HSform}\end{equation}
Here the integral is over the entire complex plane $\CC$, $L(z)$ is Lebesgue measure on $\CC$ and $\dbar$ is the 
usual d-bar operator, $\partial_x + i \partial_y$. 
Using this formula we can easily express the $r$-derivatives of $f(P)$. For example, using 
$$
\dot{(P - z)^{-1}} = -(P-z)^{-1} \dot P (P-z)^{-1},
$$
we have 
\begin{equation}\begin{gathered}
\dot {f(P)} = -\frac1{2\pi} \int_{\CC} \dbar F(z) (P - z)^{-1} \dot P  (P - z)^{-1} \, dL(z), \\
\ddot {f(P)} = \frac{2}{2\pi} \int_{\CC} \dbar F(z) (P - z)^{-1} \dot P  (P - z)^{-1} \dot P  (P - z)^{-1} \, dL(z) \\
 - \frac1{2\pi} \int_{\CC} \dbar F(z) (P - z)^{-1} \ddot P  (P - z)^{-1} \, dL(z).
\end{gathered}\label{ddotfP}\end{equation}

Now let $\mathcal{O}^k$ be the set of operators $A$, parametrized by $r$, on $L^2(Y_r)$ that are expressible in the following way: 
\begin{equation}
\int_{\CC} \partial^{j_1} \dbar F(z) Q_1(z) \dots Q_l(z) \, dL(z)
\label{Ok1}\end{equation}
where 
\begin{itemize}
\item 
each $Q_i(z)$ is either $(P-z)^{-1}$ or a multi-commutator involving $P, \dot P, \ddot P, \dots \frac{d^k}{dr^k} P$, 
\item There are a total of $j_2$ factors of $(P-z)^{-1}$ and $j_3$ commutators in the product $Q_1(z) \dots Q_l(z)$, such that
\begin{equation}
j_2 \geq 1 \text{ and } 2(j_1 +j_2 - 1) - 3 j_3 \leq k;
\label{jk1}\end{equation}
\item
the total differential order of the product $Q_i(z) \dots Q_l(z)$ is nonpositive. 
\end{itemize}
For example, the expression 
$$
\int_{\CC} \partial^{2} \dbar F(z) (P-z)^{-1} [P, [\dot P, P]] (P-z)^{-1} \ddot P (P-z)^{-2} \, dL(z)
$$
in an expression of this form with $j_1=2$, $j_2=4$ and $j_3=2$, and the total differential order is $-2$, and hence this is in $\mathcal{O}^k$ for all $k \geq 4$.

Then it is straightforward to see that $(d/dr)^k f(P)$ is in $\mathcal{O}^{2k}$ for $k = 1, 2, 3$ --- for $k = 1, 2$ it is immediate from \eqref{ddotfP}. (In fact, this is true for  all $k$, but we only need $k \leq 3$ in this proof.)

\

To complete the proof of the lemma, we show that  if $A \in \mathcal{O}^k$, then $A$ is bounded by  a multiple of $h^{-k/3}$ uniformly in $r$. 

This is proved as follows: we reorder the $Q_i$ factors in \eqref{Ok1} so that all the resolvent factors are at the left. This leaves us with commutator terms, which we discuss later. 
Then, for the expression with all the resolvent factors at the left, we integrate by parts $j_2-1$ times, obtaining an expression of the form 
$$
\int_{\CC} \partial^{j_1+j_2-1} \dbar F(z) (P-z)^{-1} \, dL(z) \circ  Q $$
where $Q$ is $h^{j_3}$ times a semiclassical differential operator, say $Q = h^{j_3} \tilde Q$. Here we use the key fact in semiclassical analysis that each time we take a commutator between two semiclassical differential operators, we gain a power of $h$ (since the order as a differential operator decreases by $1$, leaving an $h$ unpaired with a derivative).  This is equal to a constant times
$$
h^{j_3} f^{(j_1 + j_2 -1)}(P) \circ \tilde Q. 
$$
Let $2d$ be the differential order of $\tilde Q$. Then, note that $2-P(0) \geq 1 + O(h^2)$ using \eqref{P0}, hence $2 - P \geq 1 + O(h^2) + O(r)$, and so $2-P$ is therefore invertible for $r$ and $h$ small. So we can write this as
$$
h^{j_3} f^{(j_1 + j_2 -1)}(P) (2- P)^d \circ (2- P)^{-d} \tilde Q. 
$$
The operator $h^{j_3} f^{(j_1 + j_2 -1)}(P) (2- P)^d$ has an operator norm bound 
\begin{multline}
h^{j_3} \| f^{(j_1 + j_2 -1)}(P) (2- P)^d \|_{L^2(\dOmega) \to L^2(\dOmega)} \\ = h^{j_3} \sup_{t \in (-\infty, 1]} f^{(j_1 + j_2 -1)}(t) (2- t)^d  = O(h^{j_3 -2(j_1 + j_2-1)/3}). 
\label{fderivbound}\end{multline}
On the other hand, using the mapping properties of $(2-P)$ and $Q$ on semiclassical Sobolev spaces, as discussed in the Introduction, we see that $(2 - P)^{-d}\tilde Q$ is a bounded map from $L^2$ to $L^2$ with a bound uniform in $h$ as $h \to 0$. Indeed, to show this, it is sufficient to show a uniform bound on the adjoint $\tilde Q^* (2-P)^{-d}$. From  \eqref{sclSob},  we see that $(2-P)^{-d}$ maps $L^2(\dOmega)$ to $H^{2d}_h(\dOmega)$ with a bound independent of $h$, and $\tilde Q^*$ maps $H^{2d}_h(\dOmega)$ to $L^2(\dOmega)$ with a bound independent of $h$. Combining this observation with \eqref{fderivbound} and  \eqref{jk1} we see that the operator norm of this term is $O(h^{-k/3})$. 

We are left with commutator terms. The key point to note is that commuting a multicommutator term $Q_i$ past a resolvent term $(P-z)^{-1}$ gives us the factor 
$$
[(P-z)^{-1}, Q] = (P-z)^{-1} [P, Q] (P-z)^{-1}.
$$
We see that the number of resolvent factors increases by $1$, and the number of commutators increases by $1$. This means that LHS of \eqref{jk1} decreases by $1$. Moreover, the total differential order decreases by $1$. That is, all the commutator terms are in $\mathcal{O}^{k-1}$. 

By repeating this argument, we see that the commutator terms are bounded in operator norm by $O(h^{-(k-1)/3})$ up to a finite sum of double commutator terms lying in $\mathcal{O}^{k-2}$.  Thus it suffices to show that an element of $\mathcal{O}^{k-2}$ has an operator norm bound $C h^{-k/3}$. Consider the integral that yields an element of $\mathcal{O}^{k-2}$. It takes the form 
\begin{equation}
h^{j_3} \int_{\CC} \partial_z^{j_1} \dbar F(z) Q_0 (P-z)^{-1} Q_1 (P-z)^{-1} Q_3 \dots (P - z)^{-1} Q_{j_2} \, dL(z).
\label{aae-int-to-bound}\end{equation}
Here the $Q_i$ are semiclassical differential operators independent of $z$, and there are $j_2$ factors of $(P-z)^{-1}$ in all. Let $d_i$ be the differential order of $Q_i$ Since the total differential order is nonpositive, we have 
\begin{equation}
\sum_i d_i \leq 2j_2. 
\label{-2}\end{equation}
We insert powers of $(2- P)$ into this product (recall that $2-P$ is an invertible operator) as follows:
\begin{equation}\begin{gathered}
Q_0 (P-z)^{-1} Q_1 (P-z)^{-1} Q_3 \dots (P_z)^{-1} Q_{l} = \\  \Big( Q_0 (2- P)^{-a_0} \Big)  \Big( (2- P) (P-z)^{-1} \Big) \Big( (2- P)^{a_0 - 1}  Q_1 (2-P)^{-a_1} \Big) \Big( (2-P) (P-z)^{-1} \Big)  \\ \circ \,  \Big( (2- P)^{a_1-1}  Q_3 (2-P)^{-a_2} \Big)\dots \Big(  (2-P) (P-z)^{-1}\Big) \Big( (2- P)^{a_{j_2}-1} Q_{j_2} \Big)
\end{gathered}\label{expand}\end{equation}
where the $a_i$ are chosen so that each $z$-independent group $Q_0 (2- P)^{-a_0} $, $(2- P)^{a_0 - 1}  Q_1 (2-P)^{-a_1}$, etc, has nonpositive differential order; this can clearly be done by choosing the $a_0, a_1, a_2, \dots$ in turn so that each group
has differential order $0$, except for the last, which necessarily has nonpositive differential order. Thus, using the same reasoning as above (uniform boundedness on semiclassical Sobolev spaces), each $z$-independent group has an $O(1)$ operator norm  bound on $L^2(Y_r)$, uniform in $r$ and $h$. On the other hand, using spectral theory, the operator norm of $(P-z)^{-1} (2- P)$ is, for $z \in \supp F \subset [-2, 2] + i[-1, 1]$, bounded by
$$
\sup_{t \in [-\infty, 1)} (t-z)^{-1} (2- t) \leq C |\Im z|^{-1}. 
$$
Now we can bound the operator norm of \eqref{aae-int-to-bound} by substituting \eqref{expand}, using the operator norms bounds just deduced for each grouping,
together with the bound on $\dbar F$, 
$$
| \partial_z^{j_1} \dbar F(z) | \leq C |\Im z|^{j_2} h^{-2(j_1 + j_2 + 1)/3},
$$
together with the $O(h^{2/3})$ estimate on the area of the support of $F$. We get an operator bound of
$$
C h^{j_3} h^{-2(j_1 + j_2 + 1)/3}  h^{2/3} = C h^{-2(j_1+j_2-3j_3)/3} \leq C h^{-k/3},
$$
since \eqref{aae-int-to-bound} is in $\mathcal{O}^{k-2}$ and hence $2(j_1 + j_2-1) - 3j_3 \leq k-2$ according to \eqref{jk1}. 
This completes the proof.
\end{proof}

We also will need the following result.

\begin{lem}\label{lem:fPP} The operator $\dot {f(P)} P$ is uniformly bounded as $h \to 0$. 
\end{lem}

\begin{proof}
This result  seems surprising at first, since, according to Lemma~\ref{lem:F}, we only have a bound of $\sim h^{-2/3}$ on the norm of $\dot{f(P)}$. The key observation is that $\dot{f(P)}$ is (at least heuristically) spectrally supported on the support of $f'$, which is on the interval $[0, 2h^{2/3}]$, and that means that composing with $P$ gains us an additional factor of $h^{2/3}$. 

To see this rigorously, we note that by commuting operators in the first line of \eqref{ddotfP}, we
obtain an expansion for $\dot{f(P)}$ of the form
$$
\dot{f(P)} = \dot P f'(P) + [\dot P, P] f''(P) + \text{ a term in } \mathcal{O}^0.
$$
Then composing with $P$ on the right, and inserting a smooth compactly supported function $g(P)$ of $P$ where $g$  is independent of $h$ and $g = 1$ on the support of $f$, we can express 
\begin{equation}
\dot{f(P)}P = \dot P  g(P) f'(P) P + [\dot P, P] g(P) f''(P) P + \text{ a term in } \mathcal{O}^0.
\label{fPP}\end{equation}
Notice that $f'(P) P$ is a uniformly bounded family of operators, as the function $f'(t) t$ is uniformly bounded in $h$. Similarly, $f''(P) P$ is bounded in operator norm by $C h^{-2/3}$. Notice that in both of these statements, we gain $h^{2/3}$ over a naive application of Lemma~\ref{lem:F}. 

On the other hand, $\dot P g(P)$  is a uniformly bounded operator, as
we can write it as the composition of $\dot P (2-P)^{-1}$ and $(2-P) g(P)$. This first is bounded using the arguments above (as it has differential order $0$) and the second is bounded by $\sup_{t \in \RR} |(2-t) g(t)| < \infty$. Similarly, $[\dot P, P] g(P)$ is $O(h)$ as a bounded operator on $L^2(\dOmega)$. Finally, the third term on the RHS of \eqref{fPP} is $O(1)$ in operator norm by Lemma~\ref{lem:F}. Putting these statements together we see that $\dot{f(P)}$ is $O(1)$ in operator norm. 
\end{proof}

Now we prove the analogue of Lemma~\ref{lem:Energy} for the modified energy $\tE$. 

\begin{prop} \label{prop:modE}
Suppose that $u \in C^2(\overline{\Omega})$ and that $u$ satisfies the Neumann boundary condition $d_n u = 0$ at $\dOmega$. Then there is a constant $C$ depending only on $\Omega$ such that for all $h \leq 1$, the modified energy $\tE(0)$, namely $\ang{f(P(0)) P(0)v, v}$  at $\dOmega$, for $v = Uu$, is bounded by a constant times 
\begin{equation}
\Big( \| u \|_{L^2(\Omega)}^2 + h^{-2} \| (h^2 \Delta - 1) u \|_{L^2(\Omega)}^2 + \| \nabla (h^2 \Delta - 1) u \|_{L^2(\Omega)}^2 \Big) .
\label{modE-bound}\end{equation}
\end{prop}

\begin{remark}
Unlike in Lemma~\ref{lem:Energy}, we only obtain this estimate at $r=0$. We could obtain a uniform estimate for $\tE(r)$ if we had a uniform estimate on $\| h v_r \|_{L^2(Y_r)}$ for $r$ small.
\end{remark}

\begin{proof}
As before, we scale $u$ so that the RHS of \eqref{modE-bound} equals one. We then need to show that $\tE(0)$ is bounded by a constant depending only on $\Omega$. In the collar neighbourhood $\Cd$ we let $v = Uu$ as in Section~\ref{sec:1/3-bound}. Thus, we have $-(h^2\partial_r^2 + P(r) ) v = w' := Uw + h^2 b' Uu$, where $w = (h^2 \Delta - 1) u$. 

We compute the $r$-derivative of $\tE$, which we indicate by a dot: 
\begin{equation}\begin{gathered}
\dot \tE(r) = 2\ang{f(P) h^2 v_{rr}, v_r} + 2 \ang{f(P) Pv, v_r} \\
+ h^2 \ang{\dot {f(P)} v_r, v_r} +  \ang{\dot {f(P)}P v, v} + 
\ang{f(P) \dot P v, v} .
\end{gathered}\label{dottE}\end{equation}
Recalling that $w' := h^2 v_{rr} + Pv$, the  first two terms combine to give $\ang{f(P) w', v_r}$. We next calculate
\begin{equation}\begin{gathered}
\ang{\dot {f(P)} v_r, v_r} = \frac{d}{dr} \ang{\dot {f(P)} v_r, v} \\
- \ang{\ddot{f(P)} v_r, v} - \ang{\dot {f(P)} v_{rr}, v}.
\end{gathered}\label{tE2}\end{equation}
We substitute this into \eqref{dottE} and use $h^2 v_{rr} + Pv= w'$ again to obtain
\begin{equation}\begin{gathered}
\dot \tE(r) =  \ang{f(P) w', v_r} + h^2 \Big(  \frac{d}{dr} \ang{\dot {f(P)} v_r, v} - \ang{\ddot{f(P)} v_r, v} \Big) \\
-\ang{\dot f(P) w', v} + 2 \ang{\dot{f(P)} Pv, v} +      
\ang{f(P) \dot P v, v} .
\end{gathered}\label{dottE2}\end{equation}
We further calculate
\begin{equation}\begin{gathered}
\ang{\ddot{f(P)} v_r, v} = \frac1{2} \frac{d}{dr} \ang{\ddot{f(P)} v, v} \\
- \frac1{2} \ang{\dddot{f(P)} v, v}.
\end{gathered}\label{dottE2.5}\end{equation}
Substituting this into \eqref{dottE2} we find that 
\begin{equation}\begin{gathered}
\dot \tE(r) =  \ang{f(P) w', v_r} + h^2 \frac{d}{dr} \Big(  \ang{\dot {f(P)} v_r, v} + \frac1{2} \ang{\ddot {f(P)} v, v} \Big) \\
-    \frac1{2}h^2\ang{\dddot{f(P)} v, v} -\ang{\dot f(P) w', v}   + 2\ang{\dot {f(P)}P v, v} +
\ang{f(P) \dot P v, v} .
\end{gathered}\label{dottE3}\end{equation}

We now use Lemma~\ref{lem:F} to show that $E(0)$ is uniformly bounded for  $h \leq 1$. To see this we write as before, with $\eta$ as in \eqref{E2},  
\begin{equation}\begin{gathered}
\tE(0) = \int_0^\delta  \frac{d}{dr} \Big( \eta(r) \tE(r) \big)  \, dr.
\end{gathered}\label{E22}\end{equation}
Substituting our expression in for $\dot \tE$, we obtain 
\begin{equation}\begin{gathered}
\tE(0) = \int_0^\delta \eta'(r) \Big( \ang{f(P)hv_r, hv_r} + \ang{f(P) P v, v} \Big)  \, dr  \\
+ \int_0^\delta \eta(r)  h^2 \frac{d}{dr} \Big(  \ang{\dot {f(P)} v_r, v} + \ang{\ddot {f(P)} v, v} \Big)  \, dr \\
 - \int_0^\delta \eta(r) \Big(  \ang{f(P) w', v_r} -\ang{\dot f(P) w', v}  \\ +    \frac{h^2}{2} \ang{\dddot{f(P)} v, v} +  2\ang{\dot {f(P)}P v, v} +
\ang{f(P) \dot P v, v} \Big) \, dr. 
\end{gathered}\label{E23}\end{equation}
Integrating by parts in the second term, and using $\eta(0) = 1$, $\eta(\delta) = 0$, we obtain 
\begin{equation}\begin{gathered}
\tE(0) = \int_0^\delta \eta'(r) \Big( \ang{f(P)hv_r, hv_r} + \ang{f(P) P v, v} \Big)  \, dr  \\
- h^2\int_0^\delta \eta'(r)  \Big( \ang{\dot {f(P)} v_r, v} + \ang{\ddot {f(P)} v, v} \Big)  \, dr   
+ \ang{\ddot f(P)v, v} \Big|_{r=0} \\ 
 - \int_0^\delta \eta(r) \Big(    \ang{f(P) w', v_r} -\ang{\dot f(P) w', v}   \\ +  \frac{h^2}{2} \ang{\dddot{f(P)} v, v} + 2\ang{\dot {f(P)}P v, v} +
\ang{f(P) \dot P v, v} \Big) \, dr. 
\end{gathered}\label{E24}\end{equation}

Now consider each of the ten terms on the RHS of \eqref{E24} in turn. For the terms that involve no derivatives of $f$ --- that is, the first, second, sixth, and tenth terms, the argument is exactly the same as in the proof of Lemma~\ref{lem:Energy}, using the fact that 
$$
\| f(P(r)) \|_{L^2(\Omega) \to L^2(\Omega)}  = 1 \text{ for all } r \leq \delta.
$$
Te remaining terms are bounded as follows: 
\begin{itemize}
\item
The third term, $-h^2 \eta'(r) \ang{\dot{f(P)} v_r, v}$, is bounded by 
$$
h \sup_r \| \dot{f(P(r))} \|_{L^2(\dOmega) \to L^2(\dOmega)} 
 \| v \|_{\Cd} \| h\nabla v \|_{\Cd}  = O(h \times h^{-2/3} \times 1) = O(h^{1/3}),
$$
where we used Lemma~\ref{lem:F} to bound the operator norm $\dot{f(P)}$. 
\item The fourth term is bounded by 
$$
h^2 \sup_r \| \ddot{f(P(r))} \|_{L^2(\dOmega) \to L^2(\dOmega)} 
 \| v \|_{\Cd}^2 = O(h^2 \times h^{-4/3} ) = O(h^{2/3}).
$$
\item
The fifth term is similarly bounded by 
$$
h^2 \| \ddot{f(P)(0)} \|_{L^2(\dOmega) \to L^2(\dOmega)} 
 \| v \|_{L^2(\dOmega)}^2 = O(h^2 \times h^{-4/3} \times h^{-2/3} ) = O(1),
$$
where we used Proposition~\ref{1/3}. (Notice that the other boundary term, involving $\ang{\dot{f(P)} v_r, v}$, does not appear, due to the Neumann boundary condition at $r=0$.)
\item The seventh term is estimated by
$$
 \sup_r \| {\dot f(P(r))} \|_{L^2(\dOmega) \to L^2(\dOmega)} 
 \| w' \|_{\Cd} \| v \|_{\Cd} = O( h^{-2/3} \times h \times 1) = O(h^{1/3}).
$$
\item The eighth term is estimated using Lemma~\ref{lem:F} by 
$$
h^2 \sup_r \| \dddot{f(P(r))} \|_{L^2(\dOmega) \to L^2(\dOmega)} 
 \| v \|_{\Cd}^2 = O(h^2 \times h^{-2} ) = O(1).
$$
\item The ninth term is estimated using Lemma~\ref{lem:fPP}, by 
$$
 \sup_r \| \dot{f(P(r))}P(r) \|_{L^2(\dOmega) \to L^2(\dOmega)} 
 \| v \|_{\Cd}^2 = O(1).
$$
\end{itemize}

\end{proof}

\begin{remark}\label{rem:w-improvement-2} As in Remark~\ref{rem:w-improvement}, we can replace the RHS of \eqref{modE-bound} by \eqref{Tataru-bound-improv}. 
\end{remark}


\section{Upper bound}\label{sec:upperbound}
In this section, we prove an upper bound on $G_h(1 - h^2 \Deltab) u$ for $u$ an approximate eigenfunction. In fact, following the formulation of Lemma~\ref{lem:Energy} and Propositions~\ref{1/3} and \ref{prop:modE}, we express the result more generally. Using this, we obtain Theorem~\ref{thm:qo} via a $TT^*$ argument. 

\begin{prop}\label{prop:upperbound} Suppose that $u \in C^2(\overline{\Omega})$ and that $u$ satisfies the Neumann boundary condition $d_n u = 0$ at $\dOmega$. Then there is a constant $C$ depending only on $\Omega$ such that for all $h \leq 1$,  $G_h(1 - h^2 \Deltab) u$ is bounded by a constant times 
\begin{equation}
 \| u \|_{L^2(\Omega)} + h^{-1} \| (h^2 \Delta - 1) u \|_{L^2(\Omega)} + \| \nabla (h^2 \Delta - 1) u \|_{L^2(\Omega)} ,
\label{sqrt-modE-bound}\end{equation}
or indeed, using Remarks~\ref{rem:w-improvement} and \ref{rem:w-improvement-2}, by a constant times 
\begin{equation}\begin{gathered}
 \| u \|_{L^2(\Omega)} + h^{-1/2} \| u \|_{L^2(\Omega)}^{1/4} \| (h^2 \Delta - 1) u \|_{L^2(\Omega)}^{3/4} +  \\
 \Big( \| u \|_{L^2(\Omega)}^{1/2} + \| u \|_{L^2(\Omega)}^{3/8} \| (h^2 \Delta - 1) u \|_{L^2(\Omega)}^{1/8} \Big)  \| (h^2 \Delta - 1) u \|_{L^2(\Omega)}^{1/4}  \| \nabla (h^2 \Delta - 1) u \|_{L^2(\Omega)}^{1/4}
\end{gathered} .
\label{Tataru-bound-improv-2}\end{equation}
\end{prop}

\begin{proof} This is an almost  immediate consequence of Propositions~\ref{prop:modE} and Propositions~\ref{1/3}. Referring to the definition \eqref{Gh} of $G_h$, if we set $\chi = \chi_1^2$, and define $f(t) = \chi(t h^{-2/3})$, then Proposition~\ref{prop:modE} shows that 
$$
\| \chi_1(P(0) h^{-2/3} ) P(0)^{1/2} \ud \|_{L^2(\dOmega)} \leq \eqref{sqrt-modE-bound}.
$$
On the other hand, using Proposition~\ref{1/3}, and the uniform boundedness of $\chi_2(P h^{-2/3})$, we find that 
$$
h^{1/3} \| \chi_2(P(0) h^{-2/3} )  \ud \|_{L^2(\dOmega)} \leq \eqref{sqrt-modE-bound}.
$$
Since $P(0) = 1 - h^2 \Delta$, after adjusting for the $h^2 b'$ as explained in Section~\ref{sec:modE}, this completes the proof.
\end{proof}

\begin{proof}[Proof of Theorem~\ref{thm:qo}]
Let $\Pi_h$ denote the orthogonal projection onto the span of those eigenfunctions with eigenvalue $E_j$ satisfying $|\sqrt{E_j} - h^{-1}| \leq 1$. We define 
a map $T_1 : L^2(\Omega) \to L^2(\dOmega)$ by 
$$
T_1(u) = G_{h}(1-h^{2}\Delta_{\dOmega}) \big( (\Pi_h u)^{\dOmega} \big),
$$
or writing in terms of the basis of eigenfunctions, 
\begin{equation*}
T_1( \sum_j c_j u_j)=\sum_{|\sqrt{E_j} - h^{-1}| \leq 1}c_{j}G_{h}(1-h^{2}\Delta_{\dOmega})\ud_{j},
\end{equation*}
that is, 
\begin{equation}
T_1 = \sum_{|\sqrt{E_j} - h^{-1}| \leq 1} G_{h}(1-h^{2}\Delta_{\dOmega}) \ud_j \ \big\langle u_j, \cdot \big\rangle  .
\label{bndQM}\end{equation}
Then, due to the orthogonality of the $u_j$, we have 
$$T_1T_1^{\star}=\sum_{|\sqrt{E_j} - h^{-1}| \leq 1} G_{h}(1-h^2 \Deltab) \ud_j \ \Big\langle G_{h}(1-h^2 \Deltab) \ud_j, \ \cdot \ \Big\rangle$$
so it is enough to show that
\begin{equation}
\norm{T_1}_{L^2(\Omega)\to{}L^{2}(\dOmega)}\leq{}C.\label{halfest}
\end{equation}
In view of Proposition~\ref{prop:upperbound}, it suffices to show that if $u$ is in the span of eigenfunctions satisfying $|\sqrt{E_j} - h^{-1}| \leq 1$, into $L^2(\dOmega)$, and $\| u \|_{L^2(\Omega)} = 1$, then we have 
\begin{equation}
 \| u \|_{L^2(\Omega)} + h^{-1} \| (h^2 \Delta - 1) u \|_{L^2(\Omega)} + \| \nabla (h^2 \Delta - 1) u \|_{L^2(\Omega)}  \leq C.
\label{tocheck}\end{equation}
Let  $w = (h^2 \Delta - 1)u$. Then
$$w=\sum_{|\sqrt{E_j} - h^{-1}| \leq 1}c_{j}(h^{2}E_{j}-1)u$$
$$=\sum_{|\sqrt{E_j} - h^{-1}| \leq 1}c_{j}h^{2}(E_{j}^{1/2}+h^{-1})(E^{1/2}_{j}-h^{-1})u_{j}.$$
So since $|\sqrt{E_{j}}-h^{-1}|\leq{}1,$
$$\norm{(h^{2}\Delta-1)u}_{L^{2}(\Omega)}\leq{}Ch.$$
To check the gradient assumption, we write 
$$\nabla{}w=\sum_{|\sqrt{E_{j}}-h^{-1}|\leq{}1}h^{2}(E_{j}^{1/2}+h^{-1})(E^{1/2}_{j}-h^{-1})\nabla{}u_{j}$$
so
$$\norm{\nabla{}w}_{L^{2}(\Omega)}^{2}=h^{4} \!\!\!\!   \!\!\!\!  \!\!\!\! \sum_{\substack{|\sqrt{E_{j}}-h^{-1}|\leq{}1,\\ |\sqrt{E_{k}}-h^{-1}|\leq{}1}}  \!\!\!\!  \!\!\!\! (E_{j}^{1/2}+h^{-1})(E^{1/2}_{j}-h^{-1})(E_{k}^{1/2}+h^{-1})(E^{1/2}_{k}-h^{-1})\int\nabla{}u_{j}\nabla{}u_{k}.$$
However since each of the $u_{j}$ are Neumann eigenfunctions
$$\int_{\Omega}\nabla{}u_{j}\nabla{}u_{k}=-\int_{\Omega}u_{j}\Delta{}u_{k}=E_{k}\delta_{jk}$$
so
$$\norm{\nabla{}w}_{L^{2}}^{2}\leq{}Ch^{4}E^{2}\leq{}C$$
as required. This verifies \eqref{tocheck}, and completes the proof of (ii) of Theorem~\ref{Neu-est}. 
\end{proof}

\section{Exterior mass estimate}\label{sec:eme}
In the proof of Proposition~\ref{prop:modE}, we never used the fact that $f(P)$ was a positive operator. So we can apply exactly the same argument and show that the energy 
$$
\ang{f(-P) hv_r, hv_r} + \ang{f(-P)P v, v}
$$
is uniformly bounded at $r=0$, uniformly for $h < 1$.\footnote{Of course, at $r=0$, the first term vanishes due to the Neumann boundary condition.} (Here the inner product is in $L^2(\dOmega)$.)

In fact we can prove more, due to the fact that the region where $P \leq 0$ is `classically forbidden' (see the discussion in the Introduction). Let us define
$f_M(t)$ by 
$$
f_M(t) = \chi \big( -\frac{t}{M^2 h^{2/3}} \big). 
$$
Thus, $f_M$ is supported in $t \leq 0$, and it is uniformly bounded, while its derivatives $f^{(k)}$ are bounded uniformly by $C h^{-2k/3} M^{-2k}$. We think of $M$ as a large parameter. Thus, as $M \to \infty$, the spectral support of $f_M(-P)$ is further and further away from the classically allowed region. We call an estimate of (approximate) eigenfunctions exterior to the classically allowed region an `exterior mass estimate'. 

We next note that, due to the estimates on derivatives of $f$, we can modify Lemma~\ref{lem:F} to the following:

\begin{lem}\label{lem:FM}
The operator $(d/dr)^k f_M(P)(r)$ is bounded in operator norm by $$C_k h^{-2k/3} M^{-2k}$$ for $k = 0, 1, 2$, $h \leq 1$. 
\end{lem}

We omit the proof, which is straightforward to obtain from that of Lemma~\ref{lem:F} by tracing the dependence of estimates on $M$. 

We now show that $\ang{f_M(P)Pv, v}_{L^2(\dOmega)}$ tends to zero as $M \to \infty$. 

\begin{prop}\label{prop:extmass} Suppose that $u \in C^2(\overline{\Omega})$ and satisfies the Neumann boundary condition $d_n u = 0$ at $\dOmega$. 
Let $v = Uu$, as in Section~\ref{sec:1/3-bound}. Then the quantity $\ang{f_M(P)v, v}_{L^2(\dOmega)}$ is bounded by 
\begin{equation}
C h^{-2/3} M^{-3} \Big( \| u \|_{L^2(\Omega)}^2 + h^{-2} \| (h^2 \Delta - 1) u \|_{L^2(\Omega)}^2 + \| \nabla (h^2 \Delta - 1) u \|_{L^2(\Omega)}^2 \Big)
\label{Mbound1}\end{equation}
for $M \geq 1$, and $\ang{f_M(P)Pv, v}$ is bounded by 
\begin{equation}
C' M^{-1} \Big( \| u \|_{L^2(\Omega)}^2 + h^{-2} \| (h^2 \Delta - 1) u \|_{L^2(\Omega)}^2 + \| \nabla (h^2 \Delta - 1) u \|_{L^2(\Omega)}^2 \Big)
\label{Mbound2}\end{equation}
for $M \geq 1$, uniformly for $h < h_0(M)$. 
\end{prop}

\begin{proof}
To prove this, we adapt the proof of Proposition~\ref{1/3}. First,  scale $u$ so that 
\begin{equation}
 \| u \|_{L^2(\Omega)}^2 + h^{-2} \| (h^2 \Delta - 1) u \|_{L^2(\Omega)}^2 + \| \nabla (h^2 \Delta - 1) u \|_{L^2(\Omega)}^2 = 1.
\label{normalization}\end{equation}
Thus, we need to show that $\ang{f_M(P)v, v} \leq C h^{-2/3} M^{-3}$, and $\ang{f_M(P)Pv, v} \leq C' M^{-1}$.
 Let 
 $$
 L(r) = h^2 \ang{f_M(P)v, v}.
 $$
 Then, with dots indicating differentiation with respect to $r$, 
 $$
 \dot L(r) = h^2 \ang{ \dot {f_M(P)} v, v} + 2 h^2 \ang{f_M(P) v, v_r}
 $$ 
 and
 \begin{equation}
  \ddot L(r) = h^2 \ang{\ddot{f_M(P)} v, v} + 4h \ang{\dot{f_M(P)} v, hv_r} + 2\ang{f_M(P) hv_r, hv_r} + 2 \ang{f_M(P) v, h^2v_{rr}}.
 \label{L''}\end{equation}
 We write
 \begin{equation}
L(r) = L(0) + r \dot L(0) + \int_0^r ds \int_0^s dt \, \ddot L(t). 
 \label{Levolution}\end{equation}
 The key observation is that the term $\ang{f(P)v, h^2 v_{rr}}$ in \eqref{L''} is strongly positive, since $h^2 v_{rr} = - Pv+ w'$, and $-P \geq M^2 h^{2/3}$ on the support of $f_M$. We will show that, unless $L(0)$ is very small,  this term drives exponential growth of $L$ which will contradict \eqref{normalization}, by integrating $L$ on the interval $[0, h^{2/3}]$. 
 
Using the Neumann boundary condition, Lemma~\ref{lem:FM} and Proposition~\ref{1/3}, we estimate $L'(0)$ from below: 
 $$
 \dot L(0) \geq -\frac{Ch^{2/3}}{M^2}. 
 $$
 Similarly, for the terms on the RHS of \eqref{L''},  we have
 $$
|\ang{\ddot{f_M(P)} v, v}| \leq \frac{C}{M^4},$$
$$
 \quad 4h \ang{\dot{f_M(P)} v, hv_r} = 2h^2 \frac{d}{dr} \ang{\dot{f_M(P)} v, v} - 2h^2 \ang{\ddot{f_M(P)} v, v} \geq 2 h^2\frac{d}{dr} \ang{\dot{f_M(P)} v, v} -\frac{Ch^{2/3}}{M^2},
$$
and the key inequality
$$
2 \ang{f(P) v, h^2 v_{rr}} \geq 2M^2 h^{2/3} L(r) -  \ang{f_M(P)v,w'} \geq  2M^2 h^{-4/3} L(r) - h^{1/6}.
$$
Putting these together, and choosing $h$ small enough so that $h^{1/6} < M^{-4}$, we obtain
\begin{equation}
\ddot L(r) \geq -\frac{C}{M^4} + 2 h^2 \frac{d}{dr} \ang{\dot{f_M(P)} v, v} +  2M^2 h^{-4/3} L(r).
\end{equation}
Using this in \eqref{Levolution},  we get an inequality
$$
L(r) \geq L(0) - r \frac{Ch^{2/3}}{M^2} + \int_0^r ds \int_0^s dt \Big( -\frac{C}{M^4} + 2h^2 \frac{d}{dt} \ang{ \dot{f_M(P)} v, v} +2 M^2 h^{-4/3} L(t) \Big).
$$
 For $r \in [0, h^{2/3}]$, using Lemma~\ref{lem:FM} and Proposition~\ref{1/3} again, the first two terms in the big bracket can be absorbed by the $Crh^{2/3}M^{-2}$ term. We get  
 $$
L(r) \geq L(0) - r \frac{Ch^{2/3}}{M^2} + \int_0^r ds \int_0^s 2 M^2 h^{-4/3} L(t)  \, dt.
$$

It is straightforward to check that a comparison principle holds: $L(r)$ is $\geq Z(r)$ where $Z(r)$ satisfies the corresponding equality
$$
Z(r) = Z(0) - r \frac{Ch^{2/3}}{M^2} + \int_0^r ds \int_0^s  M^2 h^{-4/3} Z(t)  \, dt, \quad Z(0) = L(0). 
$$
This we can solve exactly: differentiating twice gives us
$$
\ddot Z(r) = M^2 h^{-4/3} Z(r), \quad Z(0) = L(0), \quad \dot Z(0) = -\frac{Ch^{2/3}}{M^2}.
$$
The solution is 
$$
Z(r) = L(0) \cosh (M h^{-2/3} r) - \frac{C h^{4/3}}{M^3} \sinh  (M h^{-2/3} r). 
$$
We can estimate
$$
L(r) \geq Z(r) \geq \Big( L(0) - \frac{C h^{4/3}}{M^3} \Big) e^{M h^{-2/3} r}.
$$
Now suppose, for a contradiction, that $L(0)$ was bigger than $2Ch^{4/3}M^{-3}$. Then this inequality would tell us that 
$$
L(r) \geq \frac{C h^{4/3}}{M^3}  e^{M h^{-2/3} r}, \quad r \in [0, h^{2/3}].
$$
Integrating this on $[0, h^{2/3}]$ gives 
$$
\int_0^{h^{2/3}} L(r) \, dr \geq \frac{C h^{4/3}}{M^{3}} \Big( \frac{e^{M} - 1}{M h^{-2/3}} \Big)= h^2 \frac{C(e^{M} - 1)}{M^4} ,
$$
which is bigger than $h^2 $ for large $M$, and gives us our contradiction. We conclude that 
$$
L(0) \leq \frac{2Ch^{4/3}}{M^{3}},
$$
proving \eqref{Mbound1}. 

We finally show that $\ang{f_M(P)P u, u}$ is small when $M$ is large. We estimate
\begin{equation}\begin{gathered}
\ang{f_M(P) Pu, u} = \sum_{j=0}^\infty \ang{\big(f_{M 2^j}(P) - f_{M 2^{j+1}}(P)\big)P u, u} \\
\leq \sum_{j=0}^\infty (2M^2 2^{2j} h^{2/3})\ang{f_{M 2^j}(P)  u, u} \\
\leq \sum_{j=0}^\infty (2M^2 2^{2j}h^{2/3})\frac{2Ch^{-2/3}}{M^3 2^{3j}}  \\
= 8C M^{-1} \sum_{j=0}^\infty 2^{-j} = 16C M^{-1}, 
\end{gathered}\end{equation}
proving \eqref{Mbound2}. 
\end{proof}

\begin{remark} Although not relevant to the present paper, we remark that the decay of \eqref{Mbound2} is, in fact, much faster than $M^{-1}$. Using a slightly more elaborate setup, in which we use a cutoff $f_M(t)$ that transitions between $1$ and $0$ as $t$ changes from $-M$ to $-(M+1)$, and an induction on $M \in \NN$, then we can show exponential decay of \eqref{Mbound2}. Actually, we conjecture that the decay in $M$ is even better than this, perhaps as rapid as the decay of the Airy function, i.e. $ \sim e^{-M^{3/2}}$. 

This estimate is closely related to the exterior mass estimate in \cite{MR3340369}. There, it is shown that (at least for interior hypersurfaces) if one cuts off at distance $\sim h^\delta$ from the classically allowed region, $0 < \delta < 2/3$, then the error term is $O(h^\infty)$. The method of the present paper applies to the interior hypersurface case too, and provides an endpoint result for this exterior mass estimate, in which we cut off at distance $M h^{2/3}$ and show decay as $M \to \infty$. 
\end{remark}


\section{Lower bound}\label{sec:lowerbound}

To prove the lower bound in Theorem~\ref{Neu-est}, we return to the assumption that $\Omega$ is a Euclidean domain with smooth boundary. 

\begin{prop}\label{lowerbound} Let $\Omega$ be a Euclidean domain with smooth boundary. Assume that $u \in C^3(\overline{\Omega})$ satisfies the Neumann boundary condition and is normalized in $L^2(\Omega)$. Let $w = (h^2 \Delta - 1) u$. Then there exists $\delta > 0$ and $c > 0$ depending only on $\Omega$ such that, provided 
\begin{equation}
h^{-1} \| w \| + \| \nabla w \| \leq \delta, 
\label{delta} \end{equation}
there is a uniform lower bound  
\begin{equation}
\| G_h(1 - h^2 \Deltab) \ud \|_{L^2(\partial \Omega)} \geq c,
\label{eq:lowerbound}
\end{equation}
where $G_h$ is as in \eqref{Gh}. In particular, the lower bound holds for Neumann eigenfunctions with eigenvalue $h^{-2}$. 
\end{prop}

\begin{remark} In order to get the lower bound \eqref{eq:lowerbound},  one must take $\delta$ sufficiently small. For example, consider two consecutive Neumann eigenfunctions of the disc with given angular dependence $e^{in\theta}$. These will take the form $u_1 = e^{in\theta} J_n(\lambda'_{n,l}r)$ and $u_2 = e^{in\theta} J_n(\lambda'_{n,l+1}r)$, where $\lambda'_{n,l}$
is the $l$th positive zero of $J_n'$. A suitable linear combination of these two eigenfunctions, say $u_1 + a u_2$, will vanish at the boundary, since each is a multiple of $e^{in\theta}$ there. However, with $h = 1/{\lambda'_{n,l}}$, 
$$
(h^2 \Delta - 1)(u_1 + a u_2) = a \Big( \frac{{\lambda'_{n,l+1}}^{2}}{{\lambda'_{n,l}}^{2}} - 1 \Big)  u_2  = O(h) \text{ in } L^2
$$
for fixed $n$ and $l$ large, since $\lambda_{n,l} \to \infty$ and $\lambda_{n,l+1}- \lambda_{n,l} \to \pi$ as $l \to \infty$. Similarly, we have $\| \nabla (u_1 + a u_2) \| \leq C$. Thus for sufficiently large $\delta$, we can find a sequence of $u_h$, with $h \to 0$, satisfying \eqref{delta} uniformly in $h$ and with $G_h(P) \ud$ vanishing identically. 
\end{remark}

\begin{proof}
We use a Rellich identity, which arises from the commutator of the Euclidean Laplacian with the vector field $X = x \cdot \partial_x$ generating dilations. Since $[\Delta, X] = 2 \Delta$, we find that 
\begin{equation}\begin{gathered}
2 = \| u \|_{L^2(\Omega)}^2 =  \int_{\Omega} u [h^2\Delta -1, X] u - 2 \int_{\Omega} u w  \\
= \int_{\Omega} u (h^2 \Delta - 1) (X u) - u (X w) - 2 \int_{\Omega} u w    \\
= \int_{\Omega} \Big( u (h^2 \Delta - 1) (X u) - (h^2 \Delta - 1) u (X u) \Big) + \int_{\Omega} \Big( w (Xu) - u (Xw) - 2 u w\Big)  \\
= h^2 \int_{\dOmega} \Big( u_r (X u) - u \partial_r (X u)\Big)  + \int_{\Omega} \Big( w (Xu) - u (Xw) - 2 u w\Big) .
\end{gathered}\end{equation}
The first term in the last line is zero, due to the boundary condition. Also, under the assumption \eqref{delta}, 
the integral over $\Omega$ in the last line is $O(\delta)$, following the same reasoning as in Section~\ref{sec:1/3-bound}, and hence can be absorbed in the left hand side for sufficiently small $\delta$. This yields, for $h$ sufficiently small,  
$$
\frac{3}{2} \leq  -h^2 \int_{\dOmega}  u \partial_r (X u).
$$
Consider the differential operator
$\partial_r \circ X$. This is equal to $X \circ \partial_r$ up to first order vector fields, say $V_{tan} + V_{norm}$ where $V_{tan}$ is tangent to $\dOmega$ and $V_{norm}$ is normal. We can discard $V_{norm}$ due to the boundary condition. Also, $X (u_r) = (x \cdot n) u_{rr}$, where $n$ is the incoming unit normal vector, since a tangential vector applied to $u_r$ vanishes at the boundary. Thus, we get 
$$
\frac{3}{2} \leq  h^2 \int_{\dOmega} \Big( - (x \cdot n)u  u_{rr} + u V_{tan} u \Big).
$$
With $P = P(r)$ as in Proposition~\ref{prop:modE}, we notice that we can substitute $-h^2 u_{rr} = P u + h^2 bu + w$ at the boundary. The $h^2 b u + w$ term can be absorbed in the left hand side since $\| u \|_{L^2(\dOmega)} = O(h^{-1/3})$ and $\| w \|_{L^2(\dOmega)} = O(h^{1/2})$  using the same reasoning as in Section~\ref{sec:1/3-bound}.   On the other hand, 
$$
h^2 \int_{\dOmega} u V_{\tan} u = \frac{h^2}{2} \int_{\dOmega} V_{tan}(u^2) = -\frac{h^2}{2} \int_{\dOmega} u^2 (\operatorname{div} V_{tan}) = O(h^{4/3}),
$$
so this term can also be absorbed into the LHS for small $h$. We obtain
\begin{equation}
1 \leq \int_{\dOmega} (x \cdot n) u Pu. 
\label{uPu}\end{equation}
By manipulating the RHS of \eqref{uPu}, we will obtain our lower bound. The idea is quite simple: on the classically allowed region, where $u$ is concentrated, $P$ is nonnegative. If we pretend for a moment that $P$ is nonnegative, and moreover that $P$ commutes with $x \cdot n$, then we could write the RHS as an inner product
$
\langle P^{1/2} \ud, (x \cdot n) P^{1/2} \ud \rangle 
$
which could in turn be estimated by $C \| P^{1/2} \ud \|_{L^2(\dOmega)}^2$, where $C$ is an upper bound for $x \cdot n$. This would essentially give the result as $G_h(1 - h^2 \Deltab)$ is just a regularization of $P_+^{1/2}$. 

We now give the details. 
We insert cutoffs $\Id = \chi_{in}^2(P) + \chi_{tan}^2(P) + \chi_{out}^2(P)$, acting in $L^2(\dOmega)$, where $\chi_{in}^2(t)$ is supported where $t \geq h^{2/3}$, $\chi_{tan}^2(t)$  is supported where $t \in [ - M^2h^{2/3}, h^{2/3}]$, $\chi_{out}^2(t)$ is supported where $t \leq -M^2 h^{2/3}$. For example, we can choose $\chi_{in}^2(P)$ can be of the form $\chi_1(P h^{-2/3})$ in the definition of $G$ in \eqref{Gh}; this means that $G_h(P) = \chi_{in}(P) P^{1/2} + h^{1/3} \chi_2(Ph^{-2/3})$. Notice that $\chi_2(\cdot h^{-2/3}) = \sqrt{\chi_{tan}^2 + \chi_{out}^2}$, and in particular is greater than or equal to $\chi_{tan}$. We will actually find a lower bound on the quantity 
\begin{equation}
\big\| \big( \chi_{in}(P) P^{1/2} + h^{1/3} \chi_{\tan}(P) \big) u \Big\|,
\label{Gstrong}\end{equation}
which is a stronger result than claimed in the Proposition. 

To proceed, we need the following lemma.

\begin{lem}\label{lem:comm} The commutators $[(x \cdot n), \chi_{in}(P) P^{1/2}]$  and $[(x \cdot n), \chi_{out}(P) (-P)^{1/2}]$
have operator norm bounded by $C h^{2/3}$ for small $h$. The commutator $[(x \cdot n), \chi_{tan}(P)P]$ has operator norm bounded by $C h$ for small $h$. 
\end{lem}

So as not to interrupt the flow of the argument, we postpone the proof of this Lemma until the end of the section. 

We then use this commutator estimate to obtain 
\begin{equation}\begin{aligned}
1 &\leq \ang{ (x \cdot n) Pu, u} \\
&= \bang{(x \cdot n) \big( \chi_{in}^2(P) + \chi_{tan}^2(P) + \chi_{out}^2(P) \big) Pu, u} \\
&= \bang{(x \cdot n) \chi_{in}(P) P^{1/2} u, \chi_{in}(P) P^{1/2} u} + \bang{(x \cdot n) \chi_{tan}(P) Pu, \chi_{tan}(P)u } \\  &\ \ \ - \bang{(x \cdot n)\chi_{out}(P) (-P)^{1/2} u, \chi_{out}(P) (-P)^{1/2} u} \\
&\ \ \ + \bang{[(x \cdot n), \chi_{in}(P) P^{1/2}] \chi_{in}(P) P^{1/2} u, u } + 
\bang{[(x \cdot n), \chi_{tan}^2(P) P] u, u}  \\  & \ \ \ + 
\bang{[(x \cdot n), \chi_{out}(P) (-P)^{1/2}] \chi_{out}(P) (-P)^{1/2} u, u } .
\end{aligned}\label{mess}\end{equation}
Writing $C$ for a bound on the function $|(x \cdot n)|$, we estimate the three commutator terms in the expression above. The first can be estimated, using Lemma~\ref{lem:comm}, by 
\begin{equation}\begin{gathered}
\big\| [(x \cdot n), \chi_{in}(P) P^{1/2}] \big\| \, \| \chi_{in}(P) P^{1/2} u \| \, \| u \| \\
\leq C h^{2/3} \times \sqrt{\ang{\chi_{in}^2(P) P u, u}} \times \| u \| \\
\leq C h^{2/3} \times h^0 \times h^{-1/3}  \\ = O(h^{1/3}). 
\end{gathered}\end{equation}
Here we used Cauchy-Schwarz, Proposition~\ref{1/3} and Proposition~\ref{prop:modE}  to estimate the square root factor. 
The other commutator terms are estimated similarly. So we can estimate the RHS of \eqref{mess} by 
\begin{equation}\begin{aligned}
C &\| \chi_{in}(P) P^{1/2} u \|^2 + C \|  \chi_{tan}(P) Pu \| \|  \chi_{tan}(P) u \| + C \| \chi_{out}(P) (-P)^{1/2} u \|^2 \\
&\leq C \Big(  \| \chi_{in}(P) P^{1/2} u \|^2 +  M^2 h^{2/3}  \|  \chi_{tan}(P) u \|^2 \Big) +O(M^{-1}) 
+ O(h^{1/3}) 
\end{aligned}\end{equation}
where we used the fact that $|t| \leq M^2 h^{2/3}$ on the support of $\chi_{tan}$, and  Proposition~\ref{prop:extmass} for the $\chi_{out}$ term. 

Choosing a sufficiently large value of $M$, and sufficiently small $h$ depending on $M$, we absorb the last two error terms in the LHS,  and we get a positive
lower bound on 
$$
 \| \chi_{in}(P) P^{1/2} u \|^2 +  M^2 h^{2/3}  \|  \chi_{tan}(P) u \|^2.
 $$
 Finally, since $P$ is a positive operator on the support of $\chi_{in} \chi_{tan}$, spectral theory shows that this is less than or equal to \eqref{Gstrong}, so we also get a uniform positive lower bound on \eqref{Gstrong}, which as we already mentioned, is a slightly stronger result than claimed in the statement of the Proposition. 
\end{proof}

\begin{proof}[Proof of Lemma~\ref{lem:comm}] 
Let $B$ denote the operation of multiplication by $x \cdot n$.
We prove the lemma just in the case of the first commutator, $[B, \chi_{in}(P) P^{1/2}]$. As the method of proof is very similar to the proofs of Lemmas~\ref{lem:F} and \ref{lem:fPP}, we only sketch the argument. 
To prove the estimate, we choose a smooth, compactly supported function $j(t)$ equal to $\chi_{in}(t) (1-t)^{1/2}$ 
 for $t \leq 1$, and let $J(z)$ be an almost analytic extension of $j(t)$.   Then we can expand the commutator $[B, j(P)]$ as 
$$
\int_{\CC} \dbar J(z) (P - z)^{-1} [B, P] (P - z)^{-1} dL(z)
$$
$$
= \int_{\CC} \dbar J(z) \Big( (P - z)^{-2} [B, P] + (P-z)^{-2} [[B, P], P] (P-z)^{-1} \Big)
$$
\begin{multline*}
= \int_{\CC} \dbar J(z) \Big( (P - z)^{-2} [B, P] + (P-z)^{-3} [[B, P], P] \\ + (P-z)^{-3} [[[B, P], P], P] (P-z)^{-1} \Big) dL(z) 
\end{multline*}
and so on. Similar to the proof of Lemma~\ref{lem:F}, we define 
$\mathcal{O}_J^k$ be the set of operators $A$ on $L^2(Y_r)$ that are expressible in the following way: 
\begin{equation}
\int_{\CC} \partial^{j_1} \dbar J(z) Q_1(z) \dots Q_l(z) \, dL(z)
\label{Ok}\end{equation}
where 
\begin{itemize}
\item 
each $Q_i$ is either $(P-z)^{-1}$ or a multi-commutator involving $B, P$, with $j_2$ factors of $(P-z)^{-1}$ and $j_3$ commutators, such that
\begin{equation}
j_2 \geq 1 \text{ and } 2(j_1 +j_2) - 3 j_3 \leq k;
\label{jk}\end{equation}
\item
the total differential order of the product $Q_i \dots Q_l$ is nonpositive. 
\end{itemize}
Thus, we can generate an expansion for $[B, \chi_{in}(P) P^{1/2}]$ of the form 
$$j'(P)[B, P] + j''(P)[[B, P], P] + \dots + j^{(l)}(P) [ \dots [[B, P], P] \dots , P] \text{ modulo } \mathcal{O}_J^{1-l}.
$$
As in the proof of Lemma~\ref{lem:F}, we check that the $l$th term in this expansion has operator norm $O(h^{(1+l)/3})$, while any element of $\mathcal{O}_J^{k}$ has operator norm  $O(h^{-2/3-k/3})$. That is,  the successive terms in the expansion of $[B, \chi_{in}(P) P^{1/2}]$, as well as the remainder term, improve by $h^{1/3}$.  By taking a suitable $l$ ($l = 5$ suffices) we see that $[B, \chi_{in}(P) P^{1/2}]$ has operator norm $O(h^{2/3})$. 
%
%
%
%
%
\end{proof}

\begin{remark} From the microlocal point of view, it is natural to expect that the lower bound holds on a compact Riemannian manifold, provided that the boundary `geometrically controls' the manifold in the sense of \cite{BLR}, that is, that there is a time $T$ such that every billiard trajectory of length $T$ meets the boundary non-tangentially. Indeed, one expects to be able to localize further to any open set in the boundary that geometrically controls the manifold. However, such a result certainly requires the full use of microlocal machinery, which we have elected to avoid in the present work. We hope to return to this question in a future  article. 
\end{remark}


\section{Inclusion bound}\label{sec:inclusion}


\begin{proof}[Proof of \eqref{inclusion-eval}]
We will employ much the same method as used to obtain the inclusion bound in \cite{bnds}. Clearly \eqref{inclusion-eval} holds  when $E$ is in the spectrum. For some $E$ not in the spectrum consider let $Z(E)$ be the solution operator for the Helmholtz problem
$$\begin{cases}
(\Delta-E)u=0\\
d_{n}u=G_{h}(1-h^{2}\Delta_{\dOmega})f\end{cases}, \quad h = E^{-1/2}. $$
That is $Z(E):f\to{}u$. To calculate the $L^{2}(\dOmega)\to{}L^{2}(\Omega)$ norm of $Z$ we seek to maximise the quantity
$$\frac{\norm{Z(E)[f]}_{L^{2}(\Omega)}}{\norm{f}_{L^{2}(\dOmega)}}=\frac{\norm{u}_{L^{2}(\Omega)}}{\norm{F_{h}(1-h^{2}\Delta_{\dOmega})d_{n}u}_{L^{2}(\dOmega)}}$$
but this is the same as maximising $\tilde{t}_{h}[u]^{-1}$. Therefore
$$(\min_{u}\tilde{t}_{h}[u])^{-1}=\norm{Z(E)}_{L^{2}(\dOmega)\to{}L^{2}(\Omega)}=\norm{Z(E)^{\star}}_{L^{2}(\Omega)\to{}L^{2}(\dOmega)}.$$

We will develop expressions for $Z(E)$ and $Z(E)^\star$ in terms of the Neumann eigenfunctions and use these expressions to find lower and upper bounds for $(\min_{u}\tilde{t}[u])^{-1}$. Leting $u=Z(E)[f]$, we decompose $u$ as a sum of the Neumann eigenfunctions 
$$u=\sum_{j}c_{j}u_{j}$$
To calculate the coefficients $c_{j}$ we use Green's identities along with the Helmholtz equation. We have
$$c_{j}=\langle{}u,u_{j}\rangle=\frac{1}{E-E_{j}}\int_{\Omega}(\Delta{}u)u_{j}-u(\Delta{}u_{j})$$
$$=\frac{1}{E-E_{j}}\int_{\dOmega}\ud_{j} d_{n}u-\ud d_{n}u_{j}$$
$$=\frac{1}{E-E_{j}}\langle{}\ud_{j},G_{h}(1-h^{2}\Delta_{\dOmega})f\rangle=\frac{1}{E-E_{j}}\langle{}G_{h}(1-h^{2}\Delta_{\dOmega})\ud_{j},f\rangle$$
so
$$Z(E)=\sum_{j}\frac{u_{j}\langle{}G_{h}(1-h^{2}\Delta_{\dOmega})\ud_{j},\cdot\rangle}{E-E_{j}}$$
and
\begin{equation}
Z(E)^{\star}=\sum_{j}\frac{G_{h}(1-h^{2}\Delta_{\dOmega})\ud_{j}\langle{}u_{j},\cdot\rangle}{E-E_{j}}.
\label{ZEstar}\end{equation}

We will use the operator $Z(E)^{\star}$ to obtain the inclusion bound. First let $u_{J}$ be the eigenfunction corresponding to the closest $E_{J}$ to $E$. Then
$$Z(E)^{\star}u_{J}=\frac{1}{E-E_{J}}G_{h}(1-h^{2}\Delta_{\dOmega})u_{J}^{\dOmega}$$
therefore the lower bound of Section \ref{sec:lowerbound} tells us that
$$\norm{Z(E)^{\star}u_{J}}_{L^{2}(\dOmega)}\geq{}\frac{c}{E-E_{J}}=\frac{c}{d(E,\spec_{N})}.$$
Therefore
$$\frac{1}{(\min_{u}\tilde{t}_{h}[u])}\geq{}\frac{c}{d(E,\spec_{N})}$$
or
$$c\min_{u}\tilde{t}_{h}[u]\leq{}d(E,\spec_{N})$$
which is the lower bound in \eqref{inclusion-eval}. 

To obtain the upper bound, we obtain upper bounds on the operator norm of $Z(E)^{\star}$. To obtain these upper bounds, we notice that $Z^*(E)$ can be written in the form
$$
Z^*(E) = h^2 T (h^2 \Delta^N - 1)^{-1},
$$
where $T u = G_h(1 - h^2 \Deltab) \ud$, that is, $G_h(1 - h^2 \Deltab) $ applied to the restriction of $u$ to the boundary. 
We remind the reader that $\Delta^N$ is the Neumann Laplacian. 
To estimate the norm of $Z(E)$, we use a dyadic decomposition of $L^2(\Omega)$. Let $\Pi_0$ denote the spectral projection onto eigenspaces with $E_j \leq 4E$, and for $j \geq 1$, we let $\Pi_j$ denote the spectral projection onto eigenspaces with $E_j \in (4^j E, 4^{j+1} E]$. Thus $\sum_{j \geq 0} \Pi_j = \Id$. So we write 
$$
Z^*(E) = h^2 T (h^2 \Delta^N -1)^{-1} \Big(   \sum_{j=0}^\infty \Pi_j \Big),
$$
and we denote these pieces by $Z^*_{j}$ (now dropping the $E$ dependence in notation). We also observe that $\Pi_{j}$ commutes with $(h^2 \Delta^N - 1)^{-1}$, so that we can write 
$$
Z^*_{j} = h^2 T (h^2 \Delta^N - 1)^{-1}\Pi_{j}  = h^2 T \Pi_{j} (h^2 \Delta^N - 1)^{-1} = h^2 T \Pi_{j} (h^2 \Delta^N - 1)^{-1}  \Pi_{j} .
$$

Before we begin, we observe that we always have an inequality 
\begin{equation}
d(E, \spec^N) \leq C(\Omega) \sqrt{E}. 
\label{obs}\end{equation}
To see this, choose a smooth function $\phi$, compactly supported in $\Omega$, not identically zero, and choose a unit vector $\omega \in \RR^n$. Then $u = \phi(x) e^{i\sqrt{E} x \cdot \omega}$ is in the domain of $\Delta^N$ and satisfies $\| (\Delta - E)u \|_2 \leq C \sqrt{E} \| u \|_2$, which immediately proves \eqref{obs}.  

We first consider $Z^*_{0}$. 
By Proposition~\ref{prop:upperbound}, it suffices to check that 
\begin{equation}
h^2 \| (h^2 \Delta^N - 1)^{-1} w \|_{L^2(\Omega)} + h \|  w \|_{L^2(\Omega)} + h^2 \| \nabla w \|_{L^2(\Omega)} \leq C d(E, \spec^N)^{-1} \| w \|_{L^2(\Omega)}. 
\label{ddd}\end{equation}
The first term is bounded by $d(E, \spec^N) \| w \|_{L^2(\Omega)}$ by $L^2$ spectral theory. The second term is bounded by the RHS using \eqref{obs}, and for the third term,  we argue as in Section~\ref{sec:upperbound}, showing that 
$h \| \nabla w \|_2 \leq 2 \| w \|_2$ for $u$ in  the range of $\Pi_{0}$. 

Now consider $Z^*_j$. Here, we need to use the sharp form of Proposition~\ref{prop:upperbound}, given by \eqref{Tataru-bound-improv-2}. This tells us that the norm of $Z^*_j w$ is bounded  by $Ch^2$ times 
\begin{equation}\begin{gathered}
 \| (h^2 \Delta^N - 1)^{-1} w \|_{L^2(\Omega)} + h^{-1/2} \| (h^2 \Delta^N - 1)^{-1} w \|_{L^2(\Omega)}^{1/4} \| w \|_{L^2(\Omega)}^{3/4} +  \\
 \| (h^2 \Delta^N - 1)^{-1} w \|_{L^2(\Omega)}^{1/2} \|  w \|_{L^2(\Omega)}^{1/4}  \| \nabla w \|_{L^2(\Omega)}^{1/4} \\
 + \| (h^2 \Delta^N - 1)^{-1} w \|_{L^2(\Omega)}^{3/8} \| w \|_{L^2(\Omega)}^{3/8}  \| \nabla  w \|_{L^2(\Omega)}^{1/4}.
\end{gathered} 
\label{Tataru-bound-improv-3}\end{equation}
On the range of $\Pi_j$, we can estimate the norm of $(h^2 \Delta - 1)^{-1} w$ by $C 4^{-j} \| w \|_{L^2(\Omega)}$ and the norm of $\nabla w$ by $C h^{-1} 2^{j} \| w \|_{L^2(\Omega)}$. Using this, and \eqref{Tataru-bound-improv-3}, we find that the operator norm of $Z^*_j $ is bounded  by 
$$
Ch^2 \Big( 4^{-j} + h^{-1/2} 2^{-j/2} + 2^{-j} h^{-1/4} 2^{j/4} + 2^{-3j/4} h^{-1/4} 2^{j/4} \Big) \leq C h^{3/2} 2^{-j/2}.
$$
Summing over $j$, we find that the norm of $\sum_{j \geq 1} Z_j$ is bounded by $C h^{3/2}$, which by \eqref{obs} is bounded by $C d(E, \spec^N)^{-1}$. (In this estimate we have  half a power of $h$ to spare. But the key point here is not the power of $h$, but the summability in $j$.)

Combining this with the estimate for $Z_0$ we find that the norm of $Z(E)^*$ is bounded by a constant times $ d(E, \spec^N)^{-1}$, which completes the proof. 
\end{proof}

\begin{proof}[Proof of \eqref{inclusion-efn}]
Using similar algebraic manipulations as above, we find that 
$$\begin{gathered}
u - \Pi_J u = \sum_{j \neq J} (u, u_j) u_j \\
= \sum_{j \neq J} \frac{\langle d_n u, \ud_j  \rangle  }{E - E_j} u_j \\
= \sum_{j \neq J} \frac{\big\langle F_h(1 - h^2 \Deltab)(d_n u) , G_h(1 - h^2 \Deltab) \ud_j \big\rangle }{E - E_j} u_j \\
= ( \Id - \Pi_J) \circ Z(E) \Big(F_h(1 - h^2 \Deltab)(d_n u)\Big).
\end{gathered}$$
Therefore,
\begin{equation}\begin{aligned}
\frac{\| u - \Pi_J u \|_{L^2(\Omega)}}{\| u \|_{L^2(\Omega)}} &\leq \big\| ( \Id - \Pi_J) Z(E) \big\|_{L^2(\dOmega) \to L^2(\Omega)} \frac{ \| F_h(1 - h^2 \Deltab)(d_n u) \|_{L^2(\dOmega)}}{\| u \|_{L^2(\Omega)}} \\ &=  \big\| Z^*(E) ( \Id - \Pi_J) \big\|_{L^2(\Omega) \to L^2(\dOmega)} \tttm[u].
\end{aligned}\label{opnormest}\end{equation}
So it remains to estimate the operator norm of $Z^*(E) ( \Id - \Pi_J)$. But this is done just as above. Because we have eliminated the term with $j=J$, we get the estimate above but with the distance from $E$ to the nearest point in the spectrum, $E_J$, replaced by the next nearest point, $E_*$. Thus we get 
$$
\big\| Z^*(E) ( \Id - \Pi_J) \big\|_{L^2(\Omega) \to L^2(\dOmega)}  \leq \frac{C}{|E - E_*|},
$$
and together with \eqref{opnormest} this completes the proof of \eqref{inclusion-efn}. 

\end{proof}

\begin{remark}\label{melissa}
Note that if $E$ is significantly closer (by a factor dependent on the constants in Proposition \ref{prop:upperbound} but not on $h$) to one $E_{J}$ than any other, the norm of $Z^{\star}(E)$ will be dominated by the norm of the $J$th term of the sum \eqref{ZEstar}. Indeed, in this case we have, with $E = h^{-2}$, 
$$d(E, \spec^N) \ll C(\Omega) \sqrt{E}. 
$$
Also if $P^{\perp}_{J}$ is the projection onto  the orthogonal complement of the $J$th eigenspace (i.e. projection onto the other eigenspaces)
$$
h^{2}\norm{(h^{2}\Delta^{N}-1)^{-1}P^{\perp}_{J}}_{L^{2}(\Omega)\to{}L^{2}(\Omega)} \ll d(E, \spec^N)^{-1}.$$
Combining with \eqref{ddd} this then gives
$$
\norm{Z^{\star}(E)P^{\perp}_{J}}_{L^{2}(\Omega)\to{}L^{2}(\partial\Omega)} \ll d(E, \spec^N)^{-1} $$
which, along with the lower bound implies that the $J$th term dominates the norm of $Z^{\star}(E)$. Consequently, near $E_J$, the minimum modified tension
$\tttm[\umin]$
is approximately $c_J|E - E_J|$, where the magnitude of the slope is
$c_J = \| G_h(1 - h^2 \Deltab) \ud_J \|^{-1}$.
This local `absolute value' functional form
is apparent in Figure \ref{f:tsweep}, along with the similarity of the
slope magnitudes $c_J$ ensured by Theorem~\ref{Neu-est}.
\end{remark}

\section{Numerical application of inclusion bounds}
\label{s:num}


Given a trial function $u$ and energy $E=h^{-2}$
such that the Helmholtz equation $(\Delta - E)u =0$ holds
in the domain $\Omega$,
the upper bound in Theorem~\ref{MPS-Neu}
gives a bound on the distance of $E$ to the true Neumann spectrum of the domain.
One may interpret $E$ as an approximation to a Neumann eigenvalue $E^N_j$,
and the bound $C\tilde{t}_h[u]$ as an upper
bound on the numerical error in this approximation.
Note that such a trial pair $(u,E)$ can be produced by
a variety of `global approximation' methods including the method of particular solutions
(MPS)
\cite{fhm,EnNeu,mps,Antunes2d,gsvd} and potential theoretic (boundary integral equation)
representations
\cite{coltonkress,Duran01,ungerbook,helsing_axi,Akhmet15}.
We now present and test a high-accuracy MPS algorithm
based on the modified tension \eqref{tttm-def}, which
uses boundary data alone to handle the interior norm.

\begin{figure}[t] 
\includegraphics[width=3in]{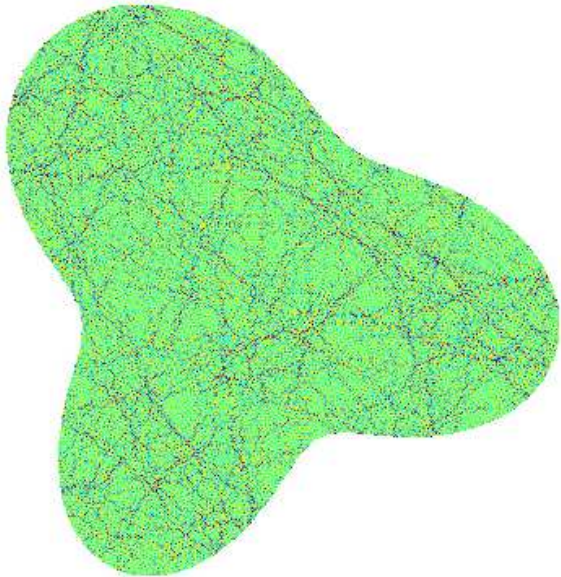}
\caption{Neumann eigenmode of the nonsymmetric smooth planar domain
of Fig.~\ref{f:intro}(b),
with eigenfrequency $\sqrt{E_j} = 405.003269518228\cdots$.
This is around the 42612th eigenvalue of the domain.
As before, the colour scale is
red positive, green zero, and blue negative.
\label{f:high}
}
\end{figure}

\subsection{High-accuracy implementation of the method of particular solutions}
\label{s:mps}

We consider the smooth planar domain shown in Fig.~\ref{f:intro}(b),
whose boundary is given in polar coordinates by
$r(\theta) = 1 + 0.3 \cos[3(\theta+ 0.2\sin\theta)]$.
We parametrize using $\theta$, ie the boundary curve in $\RR^2$ is
$x(t) = (r(t)\cos t, r(t)\sin t)$ for $t\in[0,2\pi)$.
The `method of fundamental solutions' (MFS) \cite{Bo85,Antunes2d,mfs}
provides a convenient and accurate representation for $E$-Helmholtz solutions
with smooth boundary data in a domain with smooth boundary.
Namely, we approximate
\be
u(x) \;\approx\; \sum_{n=1}^N \alpha_n \phi_n(x), \qquad x\in\Omega ~,
\label{basis}
\ee
where the fundamental solution (point source) basis functions are
\be
\phi_n(x) = Y_0(\sqrt{E}|x-y_n|), \qquad n=1,\dots,N
\label{mfs}
\ee
where $Y_0$ is the irregular Bessel function of order zero
(we choose real fundamental solutions for numerical speed).
The points $\{y_n\}_{n=1}^N$ must lie outside of $\Omega$, and their location
is reasonably important if high accuracy is desired \cite{mfs}.
We assign these points via a negative imaginary shift by $\tau$
of the boundary parametrization (a simplification of a method in \cite{mfs}):
associating $\RR^2$ with $\CC$ and analytically continuing $x(t)$
off the real $t$ axis, then $y_n =  x(2\pi n/N - i\tau)$, $n=1,\dots,N$.

At a given energy $E$ we seek the minimum modified tension
\eqref{tttm-def} over trial functions of the form \eqref{basis},
namely
\be
\tilde{t}_{h,\tbox{min}} \;:=\;
\min_{u\in \Span \{ \phi_j\}} 
\tttm[u] \;:=\;
\min_{u\in \Span \{ \phi_j\}}
\frac{ \| F_h(1-h^2 \Deltab)d_n u \|_{L^2(\dOmega)}}{ \| u \|_{L^2(\Omega)}}.
\label{tttmin}
\ee
For this we will adapt the GSVD method of Betcke \cite[Sec.~6]{gsvd}
to include both the spectral weight and the correct interior norm.

We first need a quadrature rule on $\pO$:
we choose the $M$-node periodic trapezoid rule
\be
\int_\pO f\, ds \; \approx \; \sum_{m=1}^M w_m f(x_m)
\label{ptr}
\ee
with boundary nodes $x_m = x(2\pi m/M)$ and weights
$w_m = 2\pi|x'(2\pi m/M)|/M$, which is
exponentially convergent for $f$ periodic and analytic \cite{davis59};
in practice for oscillatory integrands around 6 points per wavelength
are needed for full machine accuracy.
Then, given a vector $\alpha := \{\alpha_n\}_{n=1}^N$ of basis coefficients
in \eqref{basis},
$\| F_h(1-h^2 \Deltab)d_n u \|_{L^2(\dOmega)} \approx \| F A_\tbox{nor} \alpha\|_{l^2}$
where $A_\tbox{nor}$ is a $M$-by-$N$ matrix with elements
$$
(A_\tbox{nor})_{mn} = \sqrt{w_m} \, d_n \phi_n(x_m)~,
$$
while $F$ is an $M$-by-$M$ matrix that approximates the action of the operator
$F_h(1-h^2 \Deltab)$ on a boundary function sampled at the nodes.

We fill $F$ as follows.
Since the smoothness of cutoffs in \eqref{Gh} is not necessary
for the upper bound, we use a simpler choice of inverse spectral
weight function, $F_h(\sigma) = 1/\max[ \sigma^{1/2}_+, h^{1/3}]$.
At high boundary frequencies $|\xi'|>1$ , $F_h(1-\xi'^2)$ acts as
$h^{-1/3} \Id$, and we use projection onto a Fourier basis to
approximate the lower frequencies.
Let $P_{mn} = e^{2\pi i n s_m/L}$ for $m=1,\ldots,M$, $|n|\le n_\tbox{max} = M/4$,
be the matrix whose columns are the Fourier basis on the boundary,
where $L = |\pO|$,
and  $s_m$ is a spectral approximation to the arclength function at the
nodes, which can be computed simply using the discrete Fourier transform
of the ``speed'' vector $\{|x'(2\pi m/M)|\}_{m=1}^M$.
The $n$th Fourier mode has scaled frequency $\xi'_n = 2\pi n h / L$.
Then we evaluate $F$ via filtering of the Fourier series coefficients,
$$
F \; =\; h^{-1/3} I_M  \;+\; 
P \cdot \diag\{F_h(1-(\xi'_n)^2) - h^{-1/3}\}_{|n|\ge n_\tbox{max}} \cdot P^\ast \cdot
\diag\{w_m/L\}_{m=1}^M
~,
$$
where $I_M$ is the $M$-by-$M$ identity matrix.
The result is spectrally accurate.

In the literature, the denominator in \eqref{tttmin} is often
crudely approximated via
interior points \cite{mps,gsvd,Akhmet15}. We propose that
it is much more accurately
evaluated on the boundary using the identity \cite[Lemma~3.1]{que}
which applies to any $E$-Helmholtz function $u$,
\be
\| u\|^2_{L^2(\Omega)} \; = \; \frac{1}{2E}
\int_\pO (x\cdot n) \big(E|u|^2 + |d_n u|^2 - |d_t u|^2\big)
+ 2 \Re (x \cdot \nabla \overline{u}) d_n u \; ds
~,
\label{neuintnrm}
\ee
where $d_t$ denotes tangential derivative.
Since it is a bilinear form, $\| u\|^2_{L^2(\Omega)} = \alpha^\ast H \alpha$
for some matrix $H$.
Applying the quadrature \eqref{ptr} to \eqref{neuintnrm} gives
$$
H \approx \frac{1}{2E} \left( E A^\ast X_\tbox{nor} A +
A_\tbox{nor}^\ast X_\tbox{nor} A_\tbox{nor} - A_\tbox{tan}^\ast X_\tbox{nor} A_\tbox{tan}
+ A_\tbox{dil}^\ast A_\tbox{nor} + A_\tbox{nor}^\ast A_\tbox{dil}
\right)~,
$$
where the elements of the basis matrices are
$$
A_{mn} = \sqrt{w_m}\phi_n(x_m)~,
\quad (A_\tbox{tan})_{mn} = \sqrt{w_m}d_t \phi_n(x_m)~,
\quad (A_\tbox{dil})_{mn} = \sqrt{w_m}x_m \cdot \nabla \phi_n(x_m)
$$
and the diagonal matrix
$X_\tbox{nor} = \diag \{x_m\cdot n_m \}_{m=1}^M$, 
where $n_m$ is the normal at the $m$th node,
encodes the boundary function $x\cdot n$.

In order to apply the GSVD method, a `square-root matrix' $B$ such that
$H = B^\ast B$ is needed, so that $\|u\|_{L^2(\Omega)}\approx \|B\alpha\|_{l^2}$.
We build such a $B$ as follows.
$H$ is formally positive definite, but round-off error
in its construction means that this does not hold numerically.
Thus we diagonalize $H = V \Lambda V^\ast$, where $\Lambda = \diag\{\lambda_j\}_{j=1}^N$, then remove the columns of $V$ and $\Lambda$ for which
$\lambda_j < \eps_H \lambda_1$, where $\lambda_1$ is the largest eigenvalue,
and $\eps_H=10^{-12}$ is a cut-off not too close to machine error.
Finally we set $B = \sqrt{\Lambda} V^\ast$.

With the above matrices filled, and writing $A = F A_\tbox{nor}$,
we have for any $u$ given by a coefficient vector $\alpha$ the
tension approximation
\be
\tilde{t}_h[u] \;\approx \; \frac{\|A \alpha\|_{l^2}}{\|B \alpha\|_{l^2}}~.
\label{ratio}
\ee
Minimization of this ratio over $\alpha$ is performed by taking the
{\em generalized singular value decomposition} (GSVD)
\cite[Sec.~8.7]{golubvanloan} \cite{gsvd}.
However, since both $A$ and $B$ are usually rank-deficient, the following
regularization is needed (a simplification of that in \cite[Sec.~6]{gsvd}).
At each $E$, one first takes the SVD of $[A;B]$
(this notation indicating $A$ stacked on top of $B$),
which factorizes $[A;B] = Q \Sigma W^\ast$, where
as usual $Q$ and $W$ are unitary and $\Sigma = \diag \{\sigma_j\}_{j=1}^N$.
One keeps only the first $r_\eps$ columns of $Q$, $\Sigma$, and $W$,
where the numerical rank is $r_\eps := \#\{ j : \sigma_j \ge \eps\sigma_1\}$;
we choose the cut-off $\eps = 10^{-14}$ close to machine precision.
$Q$ is now an orthonormal basis for the numerical column space of $[A;B]$.
Splitting $Q = [Q_A;Q_B]$, we then take%
\footnote{In fact, since $Q$ has orthonormal columns, only the CS decomposition part of the GSVD is needed \cite[Sec.~8.7]{golubvanloan}.}
the GSVD of the pair $(Q_A,Q_B)$,
which decomposes $\tilde{U}^\ast Q_A X = C$ and $\tilde{V}^\ast Q_B X = S$,
where $\tilde U$, $\tilde V$ are unitary,
$C = \diag \{c_j\}_{j=1}^N$, and $S = \diag \{s_j\}_{j=1}^N$, with $c_j^2+s_j^2 =1$,
while the coefficient matrix $X$ is $r_\eps$-by-$r_\eps$ and nonsingular.
The generalized singular values are $|c_j/s_j|$;
let the index $j$ for which this is a minimum be
$j_\tbox{min}$.
Then, at this $E$, the minimum tension $\tilde{t}_h[u_\tbox{min}]$
is approximated by
\be
\tilde{t}_{h,\tbox{min}}(E) := 
\min_{\alpha\in\CC^N, \alpha\neq 0} \tilde{t}_h[u] \;\approx \;
\min_{\alpha\in\CC^N, \alpha\neq 0} \frac{\|A \alpha\|}{\|B \alpha\|}
\; = \; \frac{|c_{j_\tbox{min}}|}{|s_{j_\tbox{min}}|}
~.
\label{thmin}
\ee
The $j_\tbox{min}$th column of $X$ contains%
\footnote{Note that the definition of $X$ in MATLAB is the inverse
transpose of that in \cite[Sec.~8.7]{golubvanloan}.}
the vector
$\beta_\ast$ which minimizes $\|Q_A \beta\|/\|Q_B \beta\|$.
Transforming back to the original basis gives the
minimizing vector $\alpha_\ast = W \Sigma^{-1} \beta_\ast$.

Fig.~\ref{f:tsweep} shows resulting graphs of the function
$\tilde{t}_{h,\tbox{min}}(E)$ whose near-zeros must now be found
in order to locate approximate Neumann eigenvalues.
The graph takes the form of a (slightly rounded) absolute-value function
in the neighbourhood of each minimum.
Thus we minimize
using a combination of evaluation on a coarse grid
and iteratively fitting a parabolic approximation
to the square of $\tilde{t}_{h,\tbox{min}}(E)$ at three ordinates,
until convergence in $E$ is achieved,
as described in \cite[App.~B]{sca}.
Typically, 10--15 evaluations of \eqref{thmin}
are needed to locate each isolated minimum in $E$
to 14-digit accuracy.
At each minimum, the coefficient vector $\alpha_\ast$ is
inserted into \eqref{basis} to evaluate a trial function $u$.

\begin{remark}
We observe that, numerically,
the $E$ at which the tension is minimized and the resulting trial function $u$
depend very little on the choice of tension function used,
so it is possible to use, say, the classical tension \eqref{tclas}
for minimization, and switch to the modified tension \eqref{tttm-def} solely
for computing the final inclusion bound.
However, since it produces more uniform slopes of the tension graphs,
and requires very little extra effort to compute, we used
the modified tension throughout.
\end{remark}

\begin{table}[t] 
\begin{tabular}{|llllllll|}
\hline
frequency $\sqrt{E_j}$ & $M$ & $N$ & $\tau$ & evals & time &
$\eclas/E$ & $\enew/E$
\\
\hline
40.5128219950085 & 700 & 350 & 0.025 & 13 & 8 s &
$7.4\times 10^{-12}$ &
$3.2 \times 10^{-15}$
\\
405.003269518228 & 5000 & 2500 & 0.004 & 11 & 625 s &
$1.0\times 10^{-9}$ &
$1.0\times 10^{-14}$
\\
\hline
\end{tabular}
\vspace{3ex}
\caption{Parameters, timing and error bounds
for two Neumann eigenvalues $E_j$ computed as in Section~\ref{s:mps}.
The first row relates to the mode shown in Fig.~\ref{f:intro}(b), 
the second to the mode in Fig.~\ref{f:high}.
$M$, $N$ and $\tau$ are numerical parameters. ``evals'' is the number of
evaluations of \eqref{thmin} used to iteratively
find the minimum, taking the total CPU time shown.
$\eclas$ is the classical error bound \eqref{eclas}
for the distance from $E$ to the Neumann spectrum $\spec^N$, whereas
$\enew$ is the new proposed bound \eqref{enew},
using the constant $C=C_\tbox{est}=1.6$ in Theorem~\ref{MPS-Neu}.
Thus the last two columns show relative error bounds in eigenvalue.
\label{tbl}
}
\end{table} 

\subsection{Eigenvalue inclusion results}

We now turn to numerical results, summarized in Table~\ref{tbl}.
All computations were performed on a laptop with a quad-core i7-3720QM
2.6 GHz processor and 16 GB RAM,
running ubuntu linux, mostly in MATLAB (version R2012a).
Numerical parameters $M$, $N$, and $\tau$ giving sufficient accuracy
were found by
trial and error (with convergence in $M$ verified for fixed $N$ and $\tau$),
although it would not be hard to make these choices adaptive.
The table has two rows, corresponding to a medium and a high frequency.
The iterative minimization was started from 5 values equally
spaced in frequency intervals $\sqrt{E_j} \in [40.50, 40.55]$
and $\sqrt{E_j} \in [405, 405.005]$ respectively;
each interval is of width roughly the mean eigenfrequency spacing,
and contains a single eigenvalue.

For each trial parameter $E$ and trial function $u$ found, we compare
the classical Neumann inclusion bound to our proposed inclusion bound.
For the former, the best seem to be those of Ennenbach \cite[Thm.~7]{EnNeu},
which do not place restrictions on $u$.
This requires computing the `classical' Neumann tension
\be
\tilde{t}[u] = \frac{\| d_n u\|_{L^2(\pO)}}{\|u\|_{L^2(\Omega)}}
~,
\label{tclas}
\ee
which can be done as in Sec.~\ref{s:mps} except setting $F=I_M$.
Since $u$ is Helmholtz,
in the limit of vanishing $\|d_n u\|_{L^2(\pO)}$, Ennenbach's result simplifies to
\be
d(E,\spec^N) \;\le\; C_\tbox{Ennenbach} E \tilde{t}[u]
\;:=\; \eclas[u,E] ~.
\label{eclas}
\ee
Bounds for $C_\tbox{Ennenbach}$ are given in \cite{EnNeu}
in terms of the 2nd eigenvalue of a certain biharmonic
Stekloff problem, and geometric properties of $\Omega$.
Since our domain is star-shaped, the Stekloff eigenvalue
is bounded by geometric constants as in \cite[Sec.~3]{EnNeu},
giving, after a somewhat involved calculation, $C_\tbox{Ennenbach} = 7.4$.

From Theorem~\ref{MPS-Neu}, our proposed new upper bound is
\be
d(E,\spec^N) \;\le\; C \tttm[u] \;:=\; \enew[u,E] ~.
\label{enew}
\ee

\begin{remark}
Unlike in the Dirichlet case \cite{bnds},
the analysis in the present paper does not provide any explicit values for the
constant $C$ in Theorem~\ref{MPS-Neu}.
However, the measured slope magnitudes in the graphs of Fig.~\ref{f:tsweep}
lie in the very narrow range $[0.646, 0.676]$, strongly suggesting that
$C_\tbox{est} = 1.6$ (being an upper bound on the reciprocal) is
a valid choice for the constant. Thus we use this constant for our quoted
bounds.
\end{remark}

In the final two columns of Table~\ref{tbl} we compare
the {\em relative} eigenvalue errors implied by the two inclusion
bounds, i.e.\ $\eclas/E$ and $\enew/E$.
In the first row of Table~\ref{tbl} (matching the mode of Fig.~\ref{f:intro}(b)),
$E_j$ is around $1.6 \times 10^3$, corresponding
to eigenvalue number around $405$ (estimated via the 2-term Weyl
asymptotic).
The new bound improves over the classical bound by over 3 digits,
improving the relative error in $E_j$ from around 11 to close to 15 digits.
The high frequency case (second row, matching the mode of Fig.~\ref{f:high})
has $E_j$ around $10^2$ times higher,
and the new bound improves over the classical one by 5 digits,
improving the relative error in $E_j$ from 9 digits to 14 digits.
All digits for $\sqrt{E_j}$ given in the table are believed to be correct.

\begin{remark}
We observe that the error function $d_n u$
evaluated on the boundary nodes is dominated by uncorrelated
`noise', indicating that rounding error is the limiting factor
rather than limitations of the MFS basis representation.
Thus the frequency content of $d_n u$ falls mostly outside
of the classically allowed region, so the numerator of the tension
\eqref{tttm-def} 
is essentially $h^{-1/3} \|d_n u\|_{L^2(\pO)}$.
Note that by introducing a domain-dependent constant into the
second term of \eqref{Gh}, one might be able to improve the constant
in the bounds.
\end{remark}

Fig.~\ref{f:high} shows the corresponding
eigenmode in the high frequency case.
This is evaluated from the coefficient vector $\alpha_\ast$
by summing the basis representation \eqref{basis}--\eqref{mfs}
on a grid of 819555
points lying inside $\Omega$ in only 12.5 seconds
using the Helmholtz fast multipole (FMM) implementation
of Gimbutas--Greengard \cite{HFMM2D}.


Finally, we remind the reader that there are several global approximation
methods that, since they produce Helmholtz trial functions $u$,
can benefit from the inclusion bounds of Theorem~\ref{MPS-Neu}.
This includes the MPS, boundary integral equations, and
recent work that involves a hybrid of the two
for mixed boundary conditions \cite{Akhmet15}.
Much faster global approximation methods exist for the Dirichlet case,
by exploiting Fredholm determinants \cite{zhaodet},
or (for star-shaped domains) Neumann-to-Dirichlet operators
\cite{v+s,que,sca}.
These methods have yet to be adapted to the Neumann case.
It would be interesting to try to extend the
present analysis to domains with corners, mixed or Robin boundary conditions,
and to domain decomposition methods such as \cite{timodd}.


Code implementing the algorithms of this section,
and producing all figures in this paper,
are available in the
\verb+examples/neumann_inclusion+ directory of the
{\tt MPSpack} toolbox for MATLAB,
which can be downloaded at at the following URL:

\url{https://github.com/ahbarnett/mpspack}

\bibliographystyle{abbrv} 
\bibliography{alex}

\end{document}